\documentclass[a4paper,11pt,leqno,english]{amsart}

\usepackage[applemac]{inputenc}
\usepackage{amsmath}
\usepackage{amsthm}
\usepackage{amssymb}
\usepackage[left=2.5cm,right=2.5cm,top=3cm,bottom=3cm]{geometry}
\usepackage{amscd}
\usepackage{stmaryrd}
\usepackage[normalem]{ulem}
\usepackage{MnSymbol}

\usepackage{bbm}
\usepackage[arrow, curve, matrix]{xy}
\usepackage{faktor}

\makeatletter
\@namedef{subjclassname@2020}{\textup{2020} Mathematics Subject Classification}
\makeatother

\usepackage[cyr]{aeguill}
\usepackage{xspace}

\usepackage{tabu, multirow}

\usepackage{listings}
\usepackage{verbatim}

\usepackage{mathbbol}
\usepackage{amssymb}             

\DeclareSymbolFontAlphabet{\mathbb}{AMSb}%
\DeclareSymbolFontAlphabet{\mathbbl}{bbold}

\usepackage[table]{xcolor}
\usepackage{multirow}
\usepackage[pdftex]{graphicx}
\usepackage{enumerate}

\usepackage{array}
\usepackage{tikz-cd}
\usetikzlibrary{arrows}

\tikzset{
commutative diagrams/.cd,
arrow style=tikz,
diagrams={>=latex}}

\usepackage{hyperref}
\usepackage{varioref}
\usepackage[hyperpageref]{backref}

\usepackage{xcolor}
\hypersetup{
    colorlinks,
    linkcolor={red!60!black},
    citecolor={blue!60!black},
    urlcolor={blue!50!black}
}

\theoremstyle{thm} \newtheorem{thm}{Theorem}[section]
\theoremstyle{thm*} \newtheorem*{thm*}{Theorem}
\theoremstyle{modelthm*} \newtheorem*{modelthm*}{Model Theorem}
\newtheorem{bigthm}{Theorem}

\newtheorem{prop}[thm]{Proposition}
\newtheorem{lem}[thm]{Lemma}
\newtheorem{claim}[thm]{Claim}
\newtheorem{corol}[thm]{Corollary}
\newtheorem{conj}[thm]{Conjecture}
\theoremstyle{definition} 
\theoremstyle{remark} \newtheorem{ex}[thm]{Example}
\theoremstyle{remark} \newtheorem{rem}[thm]{Remark}
\theoremstyle{definition} \newtheorem{defi}[thm]{Definition}
\theoremstyle{definition} \newtheorem{assumption}[thm]{Assumption}

\newcommand{\quotient}[2]{{\left.\raisebox{-.2em}{$#1$}\middle\backslash\raisebox{.2em}{$#2$}\right.}}
\newcommand{\quotientd}[2]{{\left.\raisebox{.2em}{$#1$}\middle\slash\raisebox{-.2em}{$#2$}\right.}}

\newcommand{\abs}[1]{ {\left| #1 \right| } }
\newcommand{\norm}[1] {\| #1 \| }

\newcommand{\X}{{\overline{X}}}

\makeatletter
\newcommand{\xleftrightarrow}[2][]{\ext@arrow 3359\leftrightarrowfill@{#1}{#2}}
\newcommand{\xdashrightarrow}[2][]{\ext@arrow 0359\rightarrowfill@@{#1}{#2}}
\newcommand{\xdashleftarrow}[2][]{\ext@arrow 3095\leftarrowfill@@{#1}{#2}}
\newcommand{\xdashleftrightarrow}[2][]{\ext@arrow 3359\leftrightarrowfill@@{#1}{#2}}
\def\rightarrowfill@@{\arrowfill@@\relax\relbar\rightarrow}
\def\leftarrowfill@@{\arrowfill@@\leftarrow\relbar\relax}
\def\leftrightarrowfill@@{\arrowfill@@\leftarrow\relbar\rightarrow}
\def\arrowfill@@#1#2#3#4{%
  $\m@th\thickmuskip0mu\medmuskip\thickmuskip\thinmuskip\thickmuskip
   \relax#4#1
   \xleaders\hbox{$#4#2$}\hfill
   #3$%
}
\makeatother

\makeatletter
\def\blfootnote{\xdef\@thefnmark{}\@footnotetext}
\makeatother

\makeatletter
\newcommand{\thickhline}{%
    \noalign {\ifnum 0=`}\fi \hrule height 1pt
    \futurelet \reserved@a \@xhline
}
\newcolumntype{"}{@{\hskip\tabcolsep\vrule width 1pt\hskip\tabcolsep}}
\makeatother

\numberwithin{equation}{section}

\begin{document}

\author{Benoît Cadorel}
\address{Institut Élie Cartan de Lorraine \\ UMR 7502 \\ Université de Lorraine, Site de Nancy \\ B.P. 70239, F-54506 Vandoeuvre-lès-Nancy Cedex}
\email{benoit.cadorel@univ-lorraine.fr}

\author{Simone Diverio}
\address{Dipartimento di Matematica \lq\lq Guido Castelnuovo\rq\rq{}, Sapienza Università di Roma, Piazzale Aldo Moro 5, I-00185 Roma, Italy}
\email{diverio@mat.uniroma1.it}

\author{Henri Guenancia}
\address{Institut de Mathématiques de Toulouse; UMR 5219, Université de Toulouse; CNRS, UPS, 118 route de Narbonne, F-31062 Toulouse Cedex 9, France}
\email{henri.guenancia@math.cnrs.fr}


\title{On subvarieties of singular quotients of bounded domains}
\date{\today}
\begin{abstract}
Let $X$ be a quotient of a bounded domain in $\mathbb C^n$. Under suitable assumptions, we prove that every subvariety of $X$ not included in the branch locus of the quotient map is of log general type in some orbifold sense. This generalizes a recent result by Boucksom and Diverio, which treated the case of compact, étale quotients. 

Finally, in the case where $X$ is compact, we give a sufficient condition under which there exists a proper analytic subset  of $X$ containing all entire curves and all subvarieties not of general type (meant this time in in the usual sense as opposed to the orbifold sense). 
\end{abstract}
\maketitle

\section{Introduction}

Let $X$ be a quotient of a bounded domain $\Omega \subset \mathbb C^n$ by some discrete automorphism group $\Gamma \subset \mathrm{Aut}(\Omega)$. A lot of recent work has been devoted to the research of complex hyperbolicity properties of these quotients $X$, \textsl{i.e.} of restrictions on the geometry of entire curves in $X$, or on the type of its subvarieties. These quotients provide indeed basic examples to test the general conjectures in complex hyperbolicity, in particular the \emph{Green--Griffiths--Lang conjecture}:

\begin{conj} [Green, Griffiths \cite{GG80}, Lang \cite{Lan86}] \label{conjGGL} Let $M$ be a complex projective manifold of general type. Then there exists a proper algebraic subset $\mathrm{Exc}(M) \subsetneq M$ containing all the images of non constant holomorphic maps $\mathbb C \longrightarrow M$, and all the subvarieties of $M$ which are not of general type.
\end{conj}

Notably, in the case where $X = \quotient{\Gamma}{\Omega}$ is a quotient by a cocompact group acting \emph{freely} and properly discontinuously on $\Omega$, $X$ will be a complex projective manifold; it is an easy application of Liouville theorem that $X$ cannot contain any entire curve. The first statement of Conjecture~\ref{conjGGL} is trivial in this case while the second statement was recently obtained by Boucksom and the second author \cite{BD18} where they show that all subvarieties of $X$ are of general type.

When the action of $\Gamma$ is no more free or cocompact, $X$ itself may already be not of general type (cf. \textsl{e.g.} Section~\ref{examples2}) and it may not enjoy such nice hyperbolicity properties. However, there is a general philosophy that statements true in the smooth compact case should continue to hold when dealing with the correct \emph{orbifold} or \emph{logarithmic} structures in the singular or non-compact case. For a precise statement of what it is expected in the general framework of (directed) orbifolds, we refer the reader for instance to \cite[Conjecture 0.5]{CDDR21}.

Our methods yield the following archetypical result.

\begin{modelthm*}
Assume that $\overline{X}$ is a smooth projective manifold, and that $D \subset \overline{X}$ is a reduced divisor such that $X = \overline{X} \setminus D$ admits an étale cover biholomorphic to a bounded domain $\Omega \subset \mathbb C^n$. Then the pair $(\overline{X}, D)$ is of \emph{log general type}, i.e. $K_{\overline{X}} + D$ is big.
\end{modelthm*}
This theorem will appear as a particular case of several results (for instance, Theorem~\ref{thm1} or Theorem~\ref{thm2} below) whose main point of focus will be the analysis of the defect of hyperbolicity of such quotients in the spirit of the philosophy above.

Before continuing, let us point out two remarkable enough instances where our Model Theorem applies.
\begin{ex}
By a classical result of Griffiths \cite{gri71}, in any projective $n$-dimensional manifold $\overline{X}$, there exist nonempty Zariski open subsets $X = \overline{X} \setminus D$ which are uniformized by a bounded domain (moreover, of holomorphy) in $\mathbb C^n$. Even more, any point of any smooth, irreducible, quasi-projective variety over the complex numbers has a Zariski open neighborhood with this property. Thus, in principle, \emph{every} smooth projective manifold falls in the scope of the theorem above. Of course, the main interest of our result resides in situations where we have a precise description of the boundary divisor $D \subset \overline{X}$. 
\end{ex} 
Let us now give a perhaps more concrete example.
\begin{ex}\label{exmgn}
Fix integers $g\ge 2$, and $n\ge 0$, and consider the Teichm\"uller space $\mathcal T_{g,n}$ of compact Riemann surfaces of genus $g$ with $n$ marked points. It is well-known that the Bers embedding realizes $\mathcal T_{g,n}$ as a (contractible) bounded domain $\mathcal B_{g,n}\subset\mathbb C^{3g-3+n}$, which we call the \emph{Bers domain}, and actually provides an isometry between the Kobayashi distance of  $\mathcal B_{g,n}$ and the Teichm\"uller metric on $\mathcal T_{g,n}$. This domain is very far from being symmetric since it is indeed even not homogeneous: its isometry group is the Mapping Class Group $\operatorname{MCG}_{g,n}$, whose action is properly discontinuously and whose orbits are precisely the equivalent complex structures, so that the quotient $\mathcal T_{g,n}/\operatorname{MCG}_{g,n}$ is exactely the moduli space $\mathcal M_{g,n}$. 

Now, the action of $\operatorname{MCG}_{g,n}$ is not free in general (still, with finite stabilizers), but it is easy to see that it is if $n>2g+2$, so that in this case $\mathcal T_{g,n}$ is the universal cover of the smooth quasi-projective manifold $\mathcal M_{g,n}$. So therefore, in this situation our Model Theorem applies to the Deligne--Mumford compactification $\overline{\mathcal M_{g,n}}$ and gives a relatively elementary and direct proof that it is of log-general type with respect to the boundary divisor.

This (and more indeed: the canonical bundle plus the boundary divisor is not only big but even ample) was of course previously known, but to the best of our knowledge the proof relies at least on a precise description of the Picard group of $\overline{\mathcal M_{g,n}}$, together with the Cornalba--Harris ampleness criterion, which in turn employs hard GIT. 
Unfortunately, for the time being, we don't know if it is actually possible with our method to prove the ampleness of the logarithmic canonical bundle.

All this can be obtained with our Model Theorem, which is actually some sort of oversimplified statement with respect to our Theorems A, B, and C. See Examples \ref{ex:dehntwist}, \ref{ex:submfdmgn} and Remark \ref{rem:orbmgn} below for further comments on this.
\end{ex}

\begin{rem}
We would like to further observe that the Weil--Petersson metric is mapping class group invariant and thus descends to $\mathcal M_{g,n}$. It has indeed negative sectional curvature. It is known that its behaviour near the boundary gives that its Riemannian sectional curvature has as infimum negative infinity and as supremum zero. But, on the other hand, its holomorphic sectional, Ricci and scalar curvatures are all bounded above by genus-dependent negative constants. Therefore, we can also obtain the log-general type property of $\mathcal{M}_{g,n}$ by a nice application of Guenancia's theorem \cite[Theorem B]{gue18}.
\end{rem}

\medskip
 
 \subsection{Main results} The general set-up in which our results are stated involves a variety $\overline{X}$ containing a Zariski open subset $X$ admitting an étale cover $\Omega$ on which the curvature of the \emph{Bergman metric} on the canonical bundle $K_\Omega$ is positive definite at a generic point. We will call such manifolds $\Omega$ \emph{weakly Bergman manifolds}: this class of manifolds contain bounded domains, and more generally complex manifolds of \emph{bounded type} in the sense of \cite{BD18}, i.e. manifolds admitting a bounded, strictly psh function. \\


$\bullet$ \textsc{A criterion for pairs to be of log general type}

\noindent
Consider a compact K\"ahler manifold $\overline{X}$, and let $D$ be a reduced divisor on $\overline{X}$ such that $X = \overline{X} \setminus D$ is uniformized by a \emph{weakly Bergman} manifold \emph{via} a covering map $p : \Omega \longrightarrow X$. In this situation, we can then endow each component $D_i$ of $D$ with a multiplicity $m_i \in \mathbb N^\ast \cup \{\infty \}$, representing in some sense (cf. Definition \ref{defcovering}) the order of ramification of $p$ around a general point of $D_i$. We will call $\Delta_{\X} = \sum_{i} (1 - \frac{1}{m_i}) D_i$ the \emph{covering divisor} associated with our data: the pair $(\overline{X}, \Delta_\X)$ is then an orbifold pair in the sense of Campana \cite{cam04}.

Then, our first main result can be stated as follows, cf. Theorem~\ref{logbig}.

\begin{bigthm}
 \label{thm1}
Let $\overline X$ be a compact Kähler manifold endowed with a reduced divisor $D$ such that $\overline X\setminus D$ admits an étale cover biholomorphic to a weakly Bergman manifold. Let $\Delta_\X$ be the associated covering divisor on $\overline X$. 
Then, the $\mathbb Q$-line bundle $K_{\overline X}+\Delta_\X$ is big. 
\end{bigthm}

\begin{ex}\label{ex:dehntwist}
If we wanted to try to refine the situation of Example \ref{exmgn} using the more precise Theorem \ref{thm1}, we would obtain in this particular setting the same result, since the monodromy around the boundary divisor is infinite cyclic, with generator a Dehn twist, so that all the $m_i$'s are infinite here.
\end{ex}

$\bullet$ \textsc{Quotients of manifolds of bounded type}

\noindent
Let $\Omega$ be a manifold of bounded type in the sense of \cite{BD18}. We want to study the situation where $X$ is a quotient of $\Omega$ by a discrete subgroup $\Gamma \subset \mathrm{Aut}(\Omega)$ acting properly discontinuously. 

In general, $X$ is neither smooth nor compact. Given $V\subset X$ a subvariety, we explained above that one cannot expect $V$ to be of (log) general type in full generality. However, since immersed subvarieties of $\Omega$ are still \emph{weakly Bergman} by \cite{BD18}, it is then possible to apply 
Theorem~\ref{thm1} to the situation where $\overline{X}$ is replaced by a compactification $\overline{V}$ of $V$, and $\Omega$ by the fiber product $V \times_{X} \Omega$, which is still a manifold of bounded type.

The theorem below is an application of the idea above: it provides a particular setting where we can obtain that the \textit{orbifold pair} naturally associated to a modification of $V$ is of general type, cf. Theorem~\ref{newthm}.  
 
\begin{bigthm} 
\label{thm2} 
Let $\overline{X}$ be a normal, compact complex space admitting a K\"ahler resolution. Assume that it admits a Zariski open subset $X =  \overline{X} \setminus D$ which is a quotient of a manifold $\Omega$ of bounded type and let $p : \Omega \longrightarrow X$ be the quotient map. 

\noindent
Let $V \subset \overline{X}$ be a closed subvariety such that $V \not\subset D \cup \mathrm{Sing}(p)$, where $\mathrm{Sing}(p)$ is the locus of singular values of $p$. Let $\pi:\widehat{V}\to V$ be any resolution of singularities of $V$. 
  
\noindent
Then $\widehat{V}$ supports a natural covering divisor $\Delta_{\widehat{V}} = \sum_i \left( 1 - \frac{1}{n_i} \right) F_i$ such that $K_{\widehat{V}} + \Delta_{\widehat{V}}$ is big. Moreover, $\Delta_{\widehat{V}}$ is supported over $ D \cup \mathrm{Sing}(p)$ via $\pi$.
\end{bigthm}

We will actually prove a more general variant of this result, where we do not need $\widehat{V}$ to be \emph{smooth}, but merely \emph{$\mathbb Q$-factorial} (see Theorem~\ref{newthm}). This variant will be applicable to several different contexts: for example, if $X = \quotient{\Gamma}{\Omega}$ is a compact quotient and $\Gamma$ has no fixed point in codimension $1$, it will imply that $K_X$ is \emph{big} (which is not equivalent to $X$ being of general type, see Section \ref{sectexamples} (4) and Section~\ref{examples2} (1)).

\begin{rem}\label{rem:orbmgn}
Let us revisit again the situation of Example \ref{exmgn} in light of Theorem \ref{thm2}: in principle this theorem should allow to treat the more general case with no conditions on $n$, since here nothing is required about the freeness of the action. 

Given a resolution $\widehat{\mathcal M}_{g,n}$ of $\overline{\mathcal{M}_{g,n}}$, this theorem yields the existence of a natural orbifold structure $(\widehat{\mathcal M}_{g,n}, \Delta_{\widehat{\mathcal M}_{g,n}})$, which is of log-general type. Moreover, observe that the $\mathbb Q$-divisor $\Delta_{\widehat{\mathcal M}_{g,n}}$ is intrinsically completely described by the action of $\operatorname{MCG}_{g,n}$ on $\mathcal T_{g,n}$.
\end{rem}


$\bullet$ \textsc{The singular, compact case.}

\noindent
In the case where $\Omega$ is a bounded domain admitting a cocompact lattice (possibly distinct from $\Gamma$), it is possible to give a refinement of Theorem~\ref{thm2}, in the setting where $V$ is not necessarily compact, but \emph{weakly pseudoconvex}. The next theorem formulates a positivity result for $K_{\widehat{V}} + \Delta_{\widehat{V}}$ in terms of the existence of a singular metric with positive curvature on this line bundle. Applying this result to $V = \mathbb C$ will later on allow us to obtain a hyperbolicity criterion for the manifold $\overline{X}$, cf. Theorem~\ref{thmsingmetric}. 

\begin{bigthm} \label{thm3}
Assume that $\Omega$ is a bounded domain admitting a cocompact lattice. Let $q: V \longrightarrow \overline{X}$ be a generically immersive holomorphic map from a \emph{weakly pseudoconvex K\"ahler manifold} $V$, such that $q(V) \not\subset D\cup  \mathrm{Sing}(p) $. 

\noindent
Then, there exists a modification $\sigma: \widehat{V} \longrightarrow V$ such that the $\mathbb Q$-divisor $K_{\widehat{V}} + \Delta_{\widehat{V}}$ admits a singular metric $g_{\widehat{V}}$ with non-negative curvature, positive definite at a general point of $\widehat{V}$, and with an explicit lower bound on this curvature in terms of the Bergman metric of $\Omega$.
\end{bigthm}

\begin{ex}\label{ex:submfdmgn}
Getting back to the case of Example \ref{exmgn}, but still in the general situation of arbitrary $n$, we see that Theorem \ref{thm3} yields the following. 

Let $V$ be a weakly pseudoconvex K\"ahler manifold supporting a family of curves of genus $g$ with $n$ marked points, with no non-trivial automorphisms. Assume that this family has maximal variation. Then, we have a canonical generically immersive holomorphic map $f\colon V \longrightarrow {\mathcal M}_{g,n}^{\mathrm{reg}}$, and this ensures, by Theorem \ref{thm3}, that $V$ is of log-general type. This can be seen as a very special case of Viehweg's hyperbolicity conjecture.
\end{ex}

Applied to $V = \mathbb C$, this refinement can be used to provide a hyperbolicity result in the case where $X = \overline{X}$ is compact. We obtain the following partial generalization to the non-symmetric case of a previous work of the first author with Rousseau and Taji \cite{crt17}, cf. Theorem~\ref{thmsingmetric2}.  

\begin{bigthm} \label{thm4}
Assume that $\Omega$ is a bounded domain admitting a cocompact lattice. Then there exists a constant $\alpha_0$, depending only on $\Omega$, such that the following holds. 

\noindent
Let $X = \quotient{\Gamma}{\Omega}$ be a compact quotient and let $\widehat{X} \overset{\pi}{\longrightarrow} X$ be a projective resolution of singularities. If for some $\alpha>\alpha_0$, the $\mathbb Q$-divisor
$$
L_\alpha:= \pi^\ast (K_{X}+\Delta_X) - \alpha \Delta_{\widehat X} 
$$
is effective, where $\Delta_X$ is the covering divisor associated to $p:\Omega\to X$, then
\begin{enumerate}
\item any subvariety $W \subseteq X$ such that $W \not\subset \pi (\mathbb B(L_\alpha)) \cup \mathrm{Sing}(p)$ is of general type.
\item any entire curve $f: \mathbb C \longrightarrow X$ has its image included in $\pi(\mathbb B(L_\alpha)) \cup \mathrm{Sing}(p)$.
\end{enumerate}
\end{bigthm}

\noindent
Here, $\mathbb B(L_{\alpha}) = \bigcap_{p} \mathrm{Bs}(L_{\alpha}^{\otimes p})$ denotes the stable base locus of $L_{\alpha}$ and the intersection is taken over all positive integers $p$ divisible enough so that $L^{\otimes p}$ is a genuine line bundle. Note that Lemma~\ref{projective-resolution} guarantees the existence of projective resolutions for $X$ provided that it fulfils the assumptions of Theorem~\ref{thm4}. 

\subsection{Further comparison to previous results}
$\mbox{}$

\smallskip

$\bullet$ As already explained, the techniques of the papers are inspired by \cite{BD18} where it is proved that any subvariety $V\subset X$ of a compact \textit{étale} quotient $X=\quotient{\Gamma}{\Omega}$ of a manifold of bounded type is of general type. One of their key observations is that if $\widehat V\to V$ is a resolution of singularities, then one can construct a natural \textit{étale}, Galois cover $Z\to \widehat V$ where $Z$ is a \emph{Bergman manifold} i.e. the Bergman kernel on $K_Z$ is well-defined and has strictly positive curvature on a non-empty open set of $Z$; this kernel descends to define a metric with the same properties on $K_{\widehat V}$, from which the bigness of $K_{\widehat V}$ follows.  

\noindent
When the action of $\Gamma$ is not assumed to be free anymore, one can still get a Galois cover $Z\to \widehat V$ where the Bergman metric on $K_Z$ has similar positivity properties as before, but that metric will \textit{not} descend to a metric on $K_{\widehat V}$ anymore but rather on an adjoint bundle $K_{\widehat V}+\Delta_{\widehat V}$ for some suitable boundary divisor $\Delta_{\widehat V}$.
\smallskip 

$\bullet$ In the case where $\Omega$ is a bounded \emph{symmetric} domain, a great variety of viewpoints have been recently used to investigate the hyperbolicity properties of these quotients $X = \quotient{\Gamma}{\Omega}$. They can be studied by means of Hodge theory \cite{bru16, bru16a}, Monge-Ampère equations and negative holomorphic sectional curvature \cite{wy15,gue18,DT19}, or other metric methods \cite{rou15, cad16, crt17, cad18}. Unfortunately, all these techniques rely to some extent on the precise curvature properties of the Bergman metric on a bounded symmetric domain, which totally break down if the domain is not symmetric. To our knowledge, the best thing that can be said for a general bounded domain is that the holomorphic sectional curvature of its Bergman metric is bounded above by 2 \cite{Kob59} (but is has no sign in general). It would be anyway interesting to understand if the greater symmetry of bounded domains admitting a cocompact lattice might allow one to infer something more precise about the holomorphic sectional curvature of the Bergman metric (see Section \ref{criterion} where such a symmetry is exploited to obtain information on its Ricci curvature).


\subsection{Outline of the proof}
$\mbox{}$

Let us briefly describe the idea of the proof of Theorem~\ref{thm1}. Suppose that $\overline{X}$ is a compact K\"ahler manifold, and that $X = \overline{X} \setminus D$ is a Zariski open subset admitting an étale cover biholomorphic to a manifold $\Omega$ which is \emph{weakly Bergman}, \textsl{i.e.} on which the Bergman metric is defined at a generic point. We wish to find a $\mathbb Q$-divisor $\Delta_{\overline{X}}$ supported on $D$, such that $K_{\overline{X}} + \Delta_{\overline{X}}$ is big. The main idea is similar to the metric techniques employed in \cite{crt17}: we first construct a smooth metric $g$ on $K_{X}$ with positive definite curvature. Then, we control the divergence of the metric $g$ on the boundary $D$, to show it extends as a singular metric with positive curvature on $K_{\overline{X}} + \Delta_{\overline{X}}$, for some suitable $\mathbb Q$-divisor $\Delta_{\overline{X}}$ supported on $D$. The conclusion then comes from a criterion of bigness due to Boucksom \cite{bou02}.

In this situation, the metric $g$ will come directly from the Bergman kernel on $\Omega$, which descends to $X$ to define a positively curved, singular metric $g$ on $K_{X}$, with positive definite curvature at a generic point.

Finally, we have to control the divergence of $g$ near the boundary $D$. To understand this divergence, we will use a geometric construction which is very convenient to determine the adequate orbifold multiplicities to put on the components of $D$, and which was applied by several authors to extend algebraic orbifold objects on resolutions of quotient singularities (see Tai \cite{Tai1982}, Weissauer \cite{weissauer86}). \medskip

In the case where $V \subset \overline{X}$ is a subvariety of a compactification $\overline{X} = X \cup D$ of a quotient of a bounded domain $X = \quotient{\Gamma}{\Omega}$, we proceed with the same arguments, essentially replacing $X$ by a resolution of singularities $\widehat{V} \longrightarrow V$ of $V$, and $\Omega$ by $\widehat{V} \times_X \Omega$. One technical part of the proof of Theorem~\ref{thm3} is to bound from below the curvature of the Bergman metric on the open part $\widehat{V} \setminus \mathrm{Supp}( \Delta_{\widehat{V}})$: if we assume that $\Omega$ is a bounded domain acted upon by a cocompact lattice, it is possible to use general comparison results between the Carathéodory and Bergman metrics, due to Hahn \cite{hah78}. Our general method will follow closely the $L^2$ technique employed in \cite{BD18}, which was in turn inspired by \cite{CZ02}; however, it will be slightly more elaborated since we want to be able to deal with the case where $V$ is no more compact, and thus non-necessarily complete K\"ahler (see Theorem~\ref{thmsingmetric}). 

In this situation, the orbifold multiplicities to put on $\Delta_{\widehat{V}}$ give a slightly refined version of what was done in \cite{crt17}, where this technique of construction of an orbifold pair $(\widehat{V}, \Delta_{\widehat{V}})$ was also used to extend singular metrics as well as orbifold symmetric differentials. Our explicit description will allow us to compare the divisors $\Delta_{\widehat{V}}$ and $\Delta_{\widehat{X}}$ appearing in this setting, cf. Proposition~\ref{propeffective}. These comparison result will be of particular importance to prove the hyperbolicity criterion of Theorem~\ref{thm4}.

\medskip

\noindent

\subsection{Organization of the paper}
$\mbox{}$

$\bullet$ \S\!~\ref{sectcriterion}. We give a definition of the \emph{covering divisors} which will be used throughout the text. After recalling some useful information concerning the Bergman metric, and the existence of projective resolutions, we prove Theorem~\ref{thm1}. 

$\bullet$ \S\!~\ref{sectapplication}. We apply Theorem~\ref{thm1} to the case of subvarieties of quotients of bounded domains. Theorem~\ref{thm2} appears as a particular case of Theorem~\ref{newthm}, which is the main result of this section.

$\bullet$ \S\!~\ref{orbifold}. Let $V \subset X = \quotient{\Gamma}{\Omega}$, and let $\widehat{X} \longrightarrow X$ and $\widehat{V} \longrightarrow V$ be adequate log resolutions. The main result of this section is the comparison result between $\Delta_{\widehat{V}}$ and $\Delta_{\widehat{X}}$ given by Proposition~\ref{propeffective}. 

$\bullet$ \S\!~\ref{criterion}. In the case where $\Omega$ is a bounded domain, we give a uniform bound from below for the curvature of the Bergman metric of subvarieties of $\Omega$, cf. Proposition~\ref{lemcontcurvfin}. We apply this estimate to derive a lower bound of a natural singular metric with positive curvature on $K_{\widehat V}+\Delta_{\widehat V}$ in a very general setting, cf. Theorem \ref{thmsingmetric}. Finally, we go back to the compact case and spell out a criterion for $X$ to satisfy the Green--Griffiths--Lang conjecture, cf. Theorem \ref{thm4}. 

\subsection*{Acknowledgements} The authors would like to thank S\'ebastien Boucksom, Junyan Cao, Gabriele Mondello, Erwan Rousseau, Behrouz Taji, Stefano Trapani and Shengyuan Zhao for enlightening discussions about this paper. 

B.C. is partially supported by the ANR Programme: D\'efi de tous les savoirs (DS10) 2015, \lq\lq GRACK\rq\rq{}, Project ID: ANR-15-CE40-0003. 

S.D. is partially supported by the ANR Programme: D\'efi de tous les savoirs (DS10) 2015, \lq\lq GRACK\rq\rq{}, Project ID: ANR-15-CE40-0003ANR and by the ANR Programme: D\'efi de tous les savoirs (DS10) 2016, \lq\lq FOLIAGE\rq\rq{}, Project ID: ANR-16-CE40-0008. He is also partially supported by the \lq\lq Gruppo Nazionale per le Strutture Algebriche, Geometriche e le loro Applicazioni\rq\rq{} of the Istituto Nazionale di Alta Matematica \lq\lq Francesco Severi\rq\rq{} as well as the \lq\lq SEED PNR\rq\rq{} project of SAPIENZA Università di Roma

H.G. is partially supported by the National Science Foundation through the NSF Grant DMS-1510214.

\section{A criterion for a pair to be of log general type} \label{sectcriterion}

\subsection{Covering multiplicities} 
\label{subcovering}

Let us begin the present section with a definition which, although perhaps not completely standard, is well adapted to our purposes. 

\begin{defi}[Covers]
Let $X,Y$ be two irreducible and reduced complex spaces of the same dimension and let $p:X\to Y$ be a surjective holomorphic map.

$\bullet$ One says that $p$ is a \emph{cover} if there exists a discrete subgroup $\Gamma \subset \mathrm{Aut}(X)$ acting properly and discontinuously on $X$ such that $p$ is isomorphic to the quotient map $X\to \quotient{\Gamma}{X}$. 

$\bullet$ The \emph{singular locus} of $p$, $\mathrm{Sing}(p)$, is defined to be the locus of singular values of $p$, i.e. the locus of points $y\in Y$ such that there exists $x\in p^{-1}(y)$ such that $p$ is not a local biholomorphism around $x$.  

$\bullet$ If $p$ is étale, or equivalently if $\mathrm{Sing}(p)=\emptyset$, one says that $Y$ is \emph{uniformized}\footnote{Note that $X$ is not supposed to be simply connected here: we use the word \lq\lq uniformized\rq\rq{} merely to stress the fact that the cover is étale.} by $X$. 
\end{defi}

 Let $\overline X$ be a $n$-dimensional connected normal complex space and let $X\subseteq \overline X$ be some non-empty analytic Zariski open subset of $\overline X$ endowed with an étale cover $p:\widetilde X\to X$. We denote by $D=\sum_{i\in I} D_i$ the union of the codimension one irreducible components of $\overline X \setminus X$.
 
 \noindent
 Given a general point $x_i\in D_i$, one can choose an Euclidean neighborhood $\overline U_i$ of $x_i$ in $\overline X$ such that $U_i:=\overline U_i\setminus D_i \simeq \mathbb D^*\times \mathbb D^{n-1}$. One denotes by $\widetilde U_i$ a connected component of $p^{-1}(U_i)$. 

\begin{defi}[Covering divisor] \label{defcovering}
Let $(\overline X, X,p)$ as above. The \emph{covering multiplicity} $m_i=m_i(p) \in \mathbb N^*\cup \{\infty\}$ of the divisor $D_i$ is defined to be the degree of the cover $p|_{\widetilde U_i}:{\widetilde U_i} \to U_i$ above. One then defines the \emph{covering divisor} $\Delta_{\overline X}=\Delta_{\overline X}(p):=\sum_{i\in I} (1-\frac 1{m_i(p)} )D_i$.  

If no ambiguity is possible, we will just write $m_i$ and $\Delta_{\overline X}$. 
\end{defi}

\begin{rem}
The number $m_i$ above is independent of the choice of the general point $x_i\in D_i$ and the neighborhood $\overline U_i$. Moreover, it is also clearly independent of the choice of the connected component of $p^{-1}(U_i)$, since the deck transformation group acts transitively on the various components (this is because by our Definition 2.1, we always assume a cover to be Galois unless otherwise specified).
\end{rem}

\subsection{The Bergman metric} 
\label{subsectbergmanmetric}
Let $X$ be a $n$-dimensional connected complex manifold and let $\mathcal H_X$ be the Hilbert space of holomorphic sections $\sigma \in H^0(X, K_X)$ with finite $L^2$ norm
$$
\norm{\sigma}^2:= \int_{X} i^{n^2} \sigma \wedge \overline{\sigma}.
$$

Assuming that $\mathcal H_X \neq \{0\}$, one can define a singular metric $h_X$ on $K_X$ as follows. Choose a Hilbertian basis $(e_i)_i$ for $\mathcal H_X$ and let $e$ be a local holomorphic frame for $K_{X}$. For each $i$, we have $e_i = s_i e$, for some local holomorphic function $s_i$. Then, we can define
$$
\norm{e}^2_{h_{X}} = \frac{1}{\sum_i \abs{s_i}^2}.
$$

\noindent
It is easy to check that the Chern curvature current $i\Theta(h_{X})$ is a closed, positive $(1,1)$-current. Moreover, the metric $h_{X}$ is clearly invariant under the action of $\mathrm{Aut}(X)$ on $K_{X}$.

\begin{defi}[Bergman metric]
Let $X$ be a complex manifold such that  $\mathcal H_X \neq \{0\}$. 

The singular Hermitian metric $h_X$ on $K_X$ defined above is called the \textit{Bergman metric} on $X$. Its Chern curvature form $i\Theta(h_X)$ is a positive current. 
\end{defi}

\begin{rem}
Please notice that the terminology introduced above is not completely standard: usually the term Bergman metric is used for the Chern curvature form $i\Theta(h_X)$ (whenever defined), and what we call here Bergman metric is instead (a manifestation of, using the correspondence between volume forms and metrics on the canonical bundle) the Bergman kernel. 

Nevertheless, in the highly non-smooth situations we consider here, what is usually referred to as Bergman metric will be in general merely a current. Such current is very far from being a genuine smooth metric on the tangent bundle, so that we prefer to reserve the term \emph{metric} for its incarnation as a singular metric on the canonical bundle.
\end{rem}

The Bergman metric $h_X$, as well as its curvature current $i\Theta(h_X)$, are smooth on the analytic Zariski open subset of $X$ corresponding to the complement of the base locus of the linear system of $L^2$ integrable sections of $K_X$.

\begin{defi} [(weakly) Bergman manifolds] We introduce the following notions. 
\begin{enumerate}
\item A complex manifold $X$ is called \textit{weakly Bergman} if its Bergman metric $h_X$ is well-defined and if the curvature current $i\Theta(h_{X})$ is smooth and positive definite on some Euclidean open subset $\emptyset \neq U\subset X$. 
\item A \textit{Bergman manifold} is a weakly Bergman manifold $X$ where the open set $U$ above can be taken to be the whole $X$. 
\item A reduced, irreducible complex space is called weakly Bergman if $X_{\rm reg}$ is a weakly Bergman manifold. Equivalently, one (or any) resolution of $X$ is weakly Bergman.
\end{enumerate}
\end{defi}

\begin{ex} The following examples are Bergman manifolds: bounded domains in $\mathbb C^n$ and their submanifolds, bounded domains in Stein manifolds and their submanifolds (or, more generally, manifolds of bounded type, cf. \cite{BD18}), projective manifolds with very ample canonical bundle.   
\end{ex}

\begin{rem}
\label{bimero}
Let $X, Y$ be two irreducible complex spaces and let $f:X\to Y$ be a bimeromorphic map. Then $X$ is a weakly Bergman space if and only if $Y$ is a weakly Bergman space. Indeed, $f$ induces an isometry $f^*:\mathcal H_{Y_{\rm reg}} \to \mathcal H_{X_{\rm reg}}$ because any $L^2$ holomorphic $n$-form defined on an analytic Zariski open subset of a complex manifold extends automatically to the whole manifold.
\end{rem}

\begin{rem}
One can show easily that the condition  \lq\lq$i\Theta(h_{X})$ is smooth and positive definite on some open set\rq\rq{} is equivalent to the fact that $\mathcal H_X$ generates the $1$-jets at a generic point of $X$. This implies that any finite (possibly ramified) cover of a Bergman manifold is again a Bergman manifold.
\end{rem}

\subsection{Existence of projective resolutions} \label{subsectres}

The aim of this section is to prove the following technical yet useful result. 

 \begin{lem}
 \label{projective-resolution}
 Let $M$ be a Bergman manifold. Assume that there exists a discrete subgroup $\Gamma \subset \mathrm{Aut}(M)$, acting properly discontinuously and let $X := \quotient{\Gamma}{M}.$  

If $X$ is compact, then it admits a projective resolution. More generally, any compact complex space $V$ admitting a generically immersive map to $X$, whose image is not entirely contained in the singular locus of $X$, admits a projective resolution. 
 \end{lem}
 
 \begin{proof}
 We proceed in two steps. \\
 
 \noindent
 \textbf{Step 1. Case of $X$.}
 
 \noindent
 We denote by $p:M\to X$ the quotient map. The complex space $X$ is normal with quotient singularities; in particular, it is $\mathbb Q$-factorial. Moreover, there exists an effective  $\mathbb Q$-divisor $\Delta$ supported on the branch locus of $p$ such that $K_M=p^*(K_X+\Delta)$. The Bergman metric $h_M$ descends to $X$ and induces a (singular) Hermitian metric $g_X$ on $K_X+\Delta$ with positive curvature current.

Let  $\theta_{M}$ (resp. $\theta_X$) be the curvature form $i\Theta(h_{M})$ (resp. $i\Theta(g_{X})$) of $(K_{M},h_{M})$ (resp.  $(K_X+\Delta,g_X)$). The form $\theta_{M}$ is a smooth Kähler form on $M$ and one has $\theta_{M}=p^*\theta_X$. 

Let $\omega_X$ be a Hermitian metric on $X$ and let $U\Subset M$ be a relatively compact open subset of $M$ containing a fundamental domain for the action of $\Gamma$. Up to rescaling the metric $\omega_X$, one can assume that 
$$\theta_{M} \ge p^*\omega_X$$
holds on $U$. As both quantities are $\Gamma$-invariant, the inequality above is actually valid on the whole $M$. This implies that $\theta_X=p_*\theta_{M}$ is a K\"ahler current; more precisely, one has
$$\theta_X \ge \omega_X.$$

Next, we claim that the positive $(1,1)$-current $\theta_X$ has bounded local potentials. Indeed, let $V\subset X$ a small open set where $\theta_X=dd^c \phi$. On $U\cap p^{-1}(V)$, the K\"ahler form $\theta_{M}$ can be written $\theta_{M}=dd^c (p^*\phi)$, hence $p^*\phi$ is smooth and thus locally bounded on $U\cap p^{-1}(V)$. As $p$ maps that open set surjectively to $V$, our claim is proved. 

Now, let $\pi: \widehat X\to X$ be a log resolution obtained by blowing up only smooth centers. It is well-known (cf. \textsl{e.g.} \cite[Lemma~~3.5]{DP}) that there exists a smooth $(1,1)$-form $\beta\in c_1(E)$ where $E$ is a (positive) rational combination of exceptional divisors of $\pi$ such that 
$$\pi^*\theta_X-\beta$$
is a K\"ahler current on $\widehat X$. By the observation above, this K\"ahler current has vanishing Lelong numbers, hence Demailly's regularization theorem \cite{D2} enables us to find a K\"ahler form in the same cohomology class of $\pi^*\theta_X-\beta$. In particular, $\widehat X$ is K\"ahler.  Moreover, as the cohomology class of $\pi^*\theta_X-\beta$ is rational, $\widehat X$ is projective thanks to Kodaira's embedding theorem. \\

  \noindent
 \textbf{Step 2. Case of $V$.}
 
 \noindent
 Let us call $j:V\to X$ the generically immersive map from the assumptions and let us denote by $\pi:\widehat X\to X$ the projective resolution obtained in Step 1. The strict transform of $\widehat V\subset \widehat X$ of $V$ by $\pi$ is a projective variety that maps bimeromorphically to $j(V)$ by $\pi$. As a result, $j(V)$ admits a projective resolution $\widetilde V \to j(V)$. Now, the bimeromorphic map 
 $$\widetilde V \dashedrightarrow V$$
 can be resolved by a finite sequence of blow ups along smooth centers; in particular, there exists a projective manifold endowed with a surjective, proper bimeromorphic map to $V$.
 \end{proof}

\begin{rem}
When the group $\Gamma$ is linear, one can say more. Indeed, $\Gamma$ is finitely generated as being a quotient of the fundamental group of the Zariski open set $X^{\circ}\subset X$ of regular values of $p$. By Selberg's lemma, there exists a finite index subgroup $\Gamma'\subset \Gamma$ with no torsion element. As $\Gamma'$ acts properly discontinuously on $M$, the action must be free. In particular, $X':=\quotient{\Gamma'}{M}$ is smooth and $K_{X'}$ is positive by the argument above, hence $X'$ is projective. As a result, $X$ admits a finite cover by a smooth projective manifold; in particular, it is projective too.  
\end{rem}

\subsection{The criterion}
\begin{thm}
\label{logbig}
Let $\overline X$ be a compact Kähler manifold endowed with a reduced divisor $D$ such that $\overline X\setminus D$ is uniformized by a weakly Bergman manifold. Let $\Delta_{\overline X}$ be the associated covering divisor on $\overline X$. 
Then, the $\mathbb Q$-line bundle $K_{\overline X}+\Delta_{\overline X}$ is big. 
\end{thm}

\begin{rem}
An immediate corollary of the theorem is that under those assumptions, the logarithmic canonical bundle $K_{\overline X}+D$ is big. In particular, $X$ is of log general type provided that $D$ has simple normal crossings. 
\end{rem}

\begin{proof} We proceed in two steps. \\

\noindent
\textbf{Step 1. Finding a metric on $K_X$}

\noindent
Let $X:=\overline X\setminus D$ and let $p:\widetilde X\to X$ the (Galois) cover from the assumptions. The Bergman metric $h_{\widetilde X}$ is invariant under $\mathrm{Aut}(p) \subset \mathrm{Aut}(X)$ and $K_{\widetilde X}=p^*K_X$, hence it descends to a singular metric $g_X$ on $K_X$ whose curvature  $i\Theta(g_X) $ is a Kähler form on some Euclidean open set of $X$. If one can show that $g_X$ extends across $D$ as a positively curved, singular metric on $K_{\overline X}+\Delta_{\overline X}$, then Boucksom's theorem \cite[Thm.~1.2]{bou02} will show that $K_{\overline X}+\Delta_{\overline X}$ is indeed big, as expected.   \\

\noindent
\textbf{Step 2. Extending the metric to $K_{\overline X}+\Delta_{\overline X}$}

\noindent
Next, we want to analyze the behavior of $g_{X}$ near a general point of each irreducible component of $D$. Let $x\in D$ be a such a point; we denote by $m$ the covering multiplicity attached to that point (or equivalently the chosen component). There exist a small neighborhood $x\in \overline U\simeq \mathbb D^n$ in  $\overline X$ and a system of coordinates $(z_1, \ldots, z_n)$ on $\overline U$ centered at $x$ such that $\overline U\cap D =(z_1 =0)$. In particular, $ U:=\overline U \setminus D \simeq \mathbb D^*\times \mathbb D^{n-1}$. Let $\widetilde U$ a connected component of $p^{-1}(U)$.  
$$
\begin{tikzcd}
\widetilde U  \arrow[hookrightarrow]{r} \arrow{d} & \widetilde X \arrow{d}{p} & \\
 U \arrow[hookrightarrow]{r} &X   \arrow[hookrightarrow]{r} & \overline X
\end{tikzcd}
$$
By the very definition of Bergman metrics (see \cite[(4.10.4) Corollary]{kob98}), it is easy to see that one has 
$$
h_{\widetilde U} \le h_{\widetilde X},
$$
where $h_{\widetilde U}$ is the Bergman metric on (the canonical bundle of) $\widetilde U$. Moreover, that same metric is invariant under $\mathrm{Aut}(\widetilde U)$ and it descends to a metric $g_U$ on $U\simeq \Delta^* \times \Delta^{n-1}$. One has 
\begin{equation}
\label{ineqq}
g_U\le g_{X} \quad\textrm{on $U$.}
\end{equation} 
If $m=\deg(p|_{\widetilde U})$ is finite, then $p|_{\widetilde U}$ is isomorphic to 
\begin{align*}
 \Delta^* \times \Delta^{n-1} & \to  \Delta^* \times \Delta^{n-1} \\
 (w_1, \ldots, w_n) & \mapsto (w_1^m, w_2, \ldots, w_n)
\end{align*}
and otherwise it is isomorphic to the universal cover 
\begin{align*}
 \Delta^{n} & \to  \Delta^* \times \Delta^{n-1} \\
 (w_1, \ldots, w_n) & \mapsto (e^{\frac{w_1+1}{w_1-1}}, w_2, \ldots, w_n)
\end{align*}
Accordingly, one finds
\begin{equation} \label{computationnorm}
|dz_1\wedge \ldots \wedge dz_n|^{2}_{g_U}=
\begin{cases}
m^2 \cdotp |z_1|^{2(1-\frac{1}{m})}(1-|z_1|^{2/m})^2\cdotp \prod_{k=2}^n(1-|z_k|^2)^2 \\
|z_1|^2 \log^2 |z_1|^2 \cdotp  \prod_{k=2}^n(1-|z_k|^2)^2
\end{cases}
\end{equation}
By the formula above, the quantity $$-\log |dz_1\wedge \ldots \wedge dz_n|^2_{g_U}+\Big(1-\frac 1m\Big) \cdotp \log |z_1|^2$$ is locally bounded above near $\{0\}\times \mathbb D^{n-1}$ in both cases. 

Now, let us view $\sigma:=dz_1\wedge \ldots \wedge dz_n$ as an element in  $H^0(\overline U,K_{\overline X})$. By what was just said,
the psh weight 
$$
-\log |\sigma|^2_{g_{U}}+(1-\frac 1m)\log |z_1|^2\quad\textrm{on $(K_{\overline X}+\Delta_{\overline X})|_{U}$}
$$ 
is bounded above locally near $\overline U \setminus U$, hence extends across that hypersurface to a psh weight on $(K_{\overline X}+\Delta_{\overline X})|_{\overline U}$. Letting $h_{\Delta_{\overline X}}$ be a singular metric on $\mathcal O_{\overline X}(\Delta_{\overline X})$ with curvature current $[\Delta_{\overline X}]$, the process above shows that $g_{X}\otimes h_{\Delta_{\overline X}}$ extends in codimension one with positive curvature, hence everywhere. This ends the proof. \end{proof}

\begin{corol}
\label{corollary}
Let $X$ be a normal, compact complex space and let $X^{\circ}$ be some non-empty analytic Zariski open subset. Assume that one of the following conditions holds. 
\begin{enumerate}
\item[$(i)$] $X$ is a quotient of a Bergman manifold, étale over $X^{\circ}$. 
\item[$(ii)$] $X$ is $\mathbb Q$-factorial, admits a Kähler resolution and $X^{\circ}$ is uniformized by a weakly Bergman complex space. 
\end{enumerate}
Let $\Delta_{ X} \subset X\setminus X^{\circ}$ be the covering divisor associated to the above étale cover of $X^{\circ}$. Then, the $\mathbb Q$-line bundle $K_X+\Delta_X$ is big. 
\end{corol}

\begin{proof}
In both cases, $X$ admits a Kähler resolution $\pi:\widehat X\to X$.
This is a consequence of Lemma~\ref{projective-resolution} in case $(i)$ and it is an assumption in case $(ii)$. 

\noindent
By Theorem~\ref{logbig}, the covering divisor $\Delta_{\widehat X}\subset \widehat X$ associated to the étale cover of $X^{\circ} \cap X_{\rm reg} \hookrightarrow \widehat X$ by a weakly Bergman manifold satisfies that $K_{\widehat X}+\Delta_{\widehat X}$ is big. As $\pi$ is an isomorphism over the generic point of each components of $\Delta_X$, there is a $\pi$-exceptional effective $\mathbb Q$-divisor $E$ such that $\Delta_{\widehat X}=\pi^*\Delta_X+E$. In particular, one has a $\mathbb Q$-linear equivalence $K_{\widehat X}+\Delta_{\widehat X} \sim_{\mathbb Q} \pi^*(K_X+\Delta_X)$ valid over $\widehat X \setminus \mathrm{Exc}(\pi)$.  Now, the following restriction map induces an injection for any divisible enough integer $m$:
\begin{align*}
H^0(\widehat X,  m(K_{\widehat X}+\Delta_{\widehat X})) \longrightarrow & \quad H^0(\widehat X \setminus \mathrm{Exc}(\pi),  m(K_{\widehat X}+\Delta_{\widehat X})) \\
 \simeq & \quad H^0(\pi(\widehat X \setminus \mathrm{Exc}(\pi)), m (K_X+\Delta_X) ) \\
 \simeq & \quad H^0(X, m (K_X+\Delta_X) ),
\end{align*}
and it follows that $K_X+\Delta_X$ is big. 
\end{proof}

\section{Applications to quotients of manifolds of bounded type} \label{sectapplication}

\subsection{Main result}
Let $\Omega$ be a complex manifold of \emph{bounded type}, \textsl{i.e.} a complex manifold admitting a bounded strictly psh function, as defined in \cite{BD18}. This category includes bounded domains, and is stable by taking étale covers or open/closed subvarieties; the reader may wish to think of $\Omega$ as a bounded domain.

Our main object of study will be a suitable compactification of a quotient of $\Omega$. Throughout the text, we will make various assumptions on this compactification; our general hypotheses will be as follows.

\begin{assumption}
\label{generalassumption}
We fix a reduced, irreducible, compact complex space $\X$, and an open (dense) Zariski subset $X = \X\setminus D$. We assume that $X$ is a quotient of $\Omega$, \textsl{i.e.} there exists a discrete subgroup $\Gamma \subset \mathrm{Aut}(\Omega)$, acting properly discontinuously, and a fixed identification of complex spaces
$$
X = \quotient{\Gamma}{\Omega}.
$$  
We denote by $p: \Omega \longrightarrow X$ the projection map. 
\end{assumption}

 We let $X^{\circ}:= X\setminus \mathrm{Sing}(p)\subset X$ be the locus of regular values of $p$ and we set $\Omega^{\circ}:=p^{-1}(X^{\circ})$. The set $X^{\circ}$ is a Zariski-analytic open subset of $X_{\rm reg}$. Note that the inclusion $X^{\circ} \subset X_{\rm reg}$ is strict if \textsl{e.g.} $p$ ramifies in codimension one. 
 $$
\begin{tikzcd}
\Omega^{\circ}  \arrow[hookrightarrow]{r} \arrow{d} & \Omega \arrow{d}{p} &   \\
 X^{\circ} \arrow[hookrightarrow]{r} &X   \arrow[hookrightarrow]{r} & \overline X
\end{tikzcd}
$$

\begin{thm}
\label{newthm}
With the notation above, let $V$ be a normal, $\mathbb Q$-factorial compact complex space admitting a Kähler resolution and let $j:V\to \overline X$ be a generically immersive map such that $j(V)\not\subset \overline X\setminus X^{\circ}$. Set $V^{\circ}:=j^{-1}(X^{\circ})$.  

\noindent
Then, $V$ admits a natural covering divisor $\Delta_V$ supported on $V\setminus V^{\circ}$ and $K_{V}+\Delta_V$ is big.
\end{thm}

 \begin{proof}
 Let $\pi:\widehat V\to V$ be a Kähler resolution, and let $\widehat V^{\circ}:=\pi^{-1}(V^{\circ})$. Let $Z^{\circ}$ be a connected component of $p^{-1}(j(V^{\circ}))$, let $W^{\circ}:=V^{\circ}\times_{j(V^{\circ})} Z^{\circ}$ and let  $\widehat W^{\circ}:=\widehat V^{\circ}\times_{V^{\circ}} W^{\circ}$. In the following, we replace $W^{\circ}$ and $\widehat W^{\circ}$ with their irreducible component dominating $V^{\circ}$ so that the map $f:\widehat W^{\circ}\to Z^{\circ}$ below is a bimeromorphic map of irreducible complex spaces:
 
 $$
\begin{tikzcd}
\widehat W^{\circ}  \arrow[rr, bend left, "f"] \arrow{r}\arrow{d}{q} & W^{\circ} \arrow{r} \arrow{d} & Z^{\circ} \arrow[hookrightarrow]{r} \arrow{d}{p|_{Z^{\circ}}} & \Omega^{\circ} \arrow{d}{p}  \\
 \widehat V^{\circ} \arrow{r}{\pi|_{\widehat V^{\circ}}} &  V^{\circ} \arrow{r}{j|_{V^{\circ}}} & j(V^{\circ}) \arrow[hookrightarrow]{r}  & X^{\circ}.   
\end{tikzcd}
$$
As an analytic Zariski open subset of a compact Kähler manifold,  $ \widehat V^{\circ}$ can be endowed with a complete Kähler metric, cf. Lemma~\ref{complete} below. As $q$ is étale, the same property holds for the smooth manifold $\widehat W^{\circ}$. The pull back by the bimeromorphic map $f$ of the globally bounded, strictly psh function on $Z^{\circ}
$ induced by the restriction of the one living on $\Omega$ is still globally bounded, psh and strictly psh on a non-empty Zariski open subset of $\widehat W^{\circ}$. By \cite[Lemma~1.2]{BD18}, $\widehat W^{\circ}$ is a weak Bergman manifold, hence $W^{\circ}$ is a weakly Bergman complex space, cf. Remark~\ref{bimero}. The theorem now follows from Corollary~\ref{corollary} $(ii)$.
\end{proof}

%
%
%
%
%

We used the following standard result, which we recall for the reader's convenience. 
\begin{lem}
\label{complete}
Let $Y$ be a compact K\"ahler manifold and let $Y^{\circ}\subset Y$ be a Zariski open subset. There exists a complete K\"ahler metric $\omega$ on $Y^{\circ}$. 
\end{lem}

\begin{proof}
Up to taking a log resolution of $(Y,Y\setminus Y^{\circ})$ leaving $Y^{\circ}$ untouched, one can assume that the complement of $Y^{\circ}$ in $Y$ is a simple normal crossings divisor. Then, it is standard to construct Poincaré type metrics on $Y^{\circ}$, cf. \textsl{e.g.} \cite{CG}, or \cite[Th\'eor\`eme 1.5]{dem82}. 
\end{proof}

\begin{rem}
As already observed in the introduction, by the main theorem of \cite{gri71}, \emph{any projective manifold} is covered by Zariski open subsets which are uniformized by pseudoconvex bounded domains. Thus, all projective varieties fall in the scope of Assumption \ref{generalassumption}. Theorem \ref{newthm} implies in particular that for any projective manifold $\overline{X}$, there exists a Zariski open subset $X \subset \overline{X}$ such that for any subvariety $V \subset \overline{X}$, the quasi-projective variety $V \cap X$ is of log general type. 

Of course, the most interesting results will be obtained in settings where we are able to obtain a good description of this open subset $X$. 
\end{rem}

\subsection{Examples of applications of Theorem~\ref{newthm}} \label{sectexamples}
The following are the main examples of applications. 
 \begin{enumerate}
\item \textsc{[Compact étale]} 

\noindent
In the setting of Assumption~\ref{generalassumption}, assume furthermore that $\overline X=X$ is a smooth manifold and that $p$ is étale. Let $W\subset  X$ be an irreducible variety of $ X$ and let $j:V\to W$ be a resolution of singularities. Then $K_V$ is big; i.e. $W$ is of general type. This is the content of the main result of \cite{BD18}.\\

\item \textsc{[Non-compact étale]} 

\noindent
More generally, assume only that $\overline X$ is a smooth compact Kähler manifold and that $p$ is étale. Let $W\subset \overline X$ be a compact subvariety not included in $\overline X\setminus X$ and let $j:V\to W$ be a log resolution of $(W,W\cap (\overline X\setminus X))$. Then there exists a divisor $\Delta_V$ on $V$, supported over $\overline X\setminus X$ via $j$ such that $K_V+\Delta_V$ is big. In particular, $W \cap X$ is of log general type. \\

\item \textsc{[Compact non-étale]} 

\noindent
Assume that $\overline X=X$, let $W\subset X$ be an irreducible subvariety not included in the branch locus of $p$ and let $j:V\to W$ be a resolution of $W$. Then, there exists a natural divisor $\Delta_V$ on $V$ with coefficients in $(0,1)$, supported over the branch locus of $p$ via $j$ such that $K_V+\Delta_V$ is big. \\

\item \textsc{[$\mathbb Q$-factorial subvarieties]}

\noindent
In the setting~\ref{generalassumption}, assume that $\overline X$ admits a Kähler resolution. Let $V\subset \overline X$ be a normal, $\mathbb Q$-factorial subvariety not included in the branch locus of $p$ or $\overline X \setminus X$. Then, there exists a reduced divisor $\Delta_V$ on $V$, supported on $\overline X\setminus X^\circ$, such that $K_V+\Delta_V$ is big.\\

 \item \textsc{[$\mathbb Q$-factorialization of subvarieties]} \label{qfactlabel}
 
 \noindent
 Assume that $\overline X=X$, let $W\subset X$ be an irreducible closed subvariety not included in the branch locus of $p$, let $\widetilde W$ be a connected component of $p^{-1}(W)$ and let $\Gamma_{\widetilde W} \subset \Gamma$ be the subgroup of $\Gamma$ preserving $\widetilde W$.  Let $\widetilde V\to \widetilde W$ be a $\Gamma_{\widetilde W}$-equivariant resolution of singularities of $\widetilde{W}$ and let $V:= \quotient{\Gamma_{\widetilde W}}{\widetilde V}.$ This normal variety has quotient singularities, in particular it is $\mathbb Q$-factorial, and it comes equipped with a birational map $j:V\to W$. 
  $$
\begin{tikzcd}
\widetilde V \arrow{r} \arrow{d} &\widetilde W \arrow{d}  \arrow[hookrightarrow]{r}  & \Omega \arrow{d}{p}    \\
 V \arrow{r}{j} &W   \arrow[hookrightarrow]{r} &  X
\end{tikzcd}
$$

 \noindent
 Moreover,  there is a natural branching divisor $\Delta_V$ on $V$ attached to the cover $\widetilde V \to V$, supported over $\mathrm{Sing}(p)$ via $j$ and satisfying that $K_{V}+\Delta_V$ is big.

 \end{enumerate}

 \subsection{Two remarks about the singular vs smooth case}
 \label{examples2}
 The following two examples show that the situation in the singular case is rather subtle. 
 
 \begin{enumerate}
\item \textsc{[$K_X$ big v. $X$ of general type]} 

\noindent
Assume for simplicity that $X$ is compact and $p$ is quasi-étale, that is, $\mathrm{codim}_X(X\setminus X^\circ)\ge 2$. Then, Example (4) in Section~\ref{sectexamples} shows that $K_X$ is big. Unless $X$ has only canonical singularities, this property is weaker than saying that $X$ is of general type, \textsl{i.e.} that the canonical bundle $K_{\widehat X}$ of a (or any) resolution $\widehat X\to X$ is big. 

\noindent
For instance, there exist surfaces $S$ which are a quotient of the bi-disk $\Delta^2\subset \mathbb C^2$ such that $K_S$ is ample and yet $S$ is not of general type. One can realize such surfaces as  $S=(C_1\times C_2)/G$ where $C_1,C_2$ are curves of genus at least two and $G$ is a finite group acting diagonally, cf. \cite[Table~1]{BP16}. \\

\item \textsc{[$X$ of general type but not all its subvarieties]} 

\noindent
Assume again that $X$ is compact and $p$ is quasi-étale. The example above shows that $X$ needs not be of general type. Even if we assume that $X$ is of general type, it may still happen that $X$ contains subvarieties $V\not \subset X_{\rm sing}$ such that $V$ \textit{is not} of general type. 

\noindent
Indeed, let $C$ be a hyperelliptic curve and let $f:C\to \mathbb P^1$ be the double cover; it induces an involution $\iota \in \mathrm{Aut}(C)$. The transformation $C\times C \ni (z,w) \mapsto (\iota(w),\iota(z))$ induces an action of $\quotientd{\mathbb Z}{4 \mathbb Z} $ on $C\times C$. Let $X:= \quotient{ \quotientd{\mathbb Z}{4\mathbb Z}}{C\times C}$; it is a projective variety with canonical singularities  and ample canonical bundle admitting a cover by the bidisk in $\mathbb C^2$. Yet, the diagonal map $C\to X$ factors through $\mathbb P^1$ as showed below.
 \begin{figure}[h!]
 \centering
\begin{tikzcd}
C \ar[hookrightarrow]{r}{\Delta}\ar["f"']{rd} & C\times C \ar{r} &X \\
 & \mathbb P^1 \ar["j"']{ru} & 
\end{tikzcd}
\end{figure}

\end{enumerate}

\section{Comparison of covering divisors}
\label{orbifold}
In this section, we work under the general Assumption \ref{generalassumption}. Given any generically immersive map $V \longrightarrow \overline{X}$, Definition \ref{defcovering} allows us to attach a natural covering divisor $\Delta_{\widehat{V}}$ to any resolution of singularities of $\widehat{V} \longrightarrow V$. In this section, we will gather a few facts allowing us to compare the natural orbifold structures on adequate resolution of singularities of $V$ and $\overline{X}$.  

Let us recall how Definition \ref{defcovering} permits to construct the orbifold structures in this context.

\subsection{Natural orbifold structure on a resolution of singularities of a singular quotient}
\label{construction}

Let us fix a log-resolution $\widehat{X} \overset{\pi}{\longrightarrow} \overline{X}$, such that the preimage of the union of $D$ and the closure of $\mathrm{Sing}(p)$ is a divisor $E$ with simple normal crossings. Let $E = \sum_i E_i$ be the decomposition of $E$ into its irreducible components. Also, let $g_{X}$ be the natural smooth metric induced on $K_{X}$ by the Bergman kernel on $K_{\Omega}$. 

\medskip

Let $Y$ be the normalization of the component $T_{\widehat X}$ of the fiber product $\widehat{X} \times_{\overline{X}} \Omega$ dominating $\widehat X$ (or, equivalently, $\overline X$). It sits in the following commutative diagram.
 \begin{figure}[h]
 \centering
\begin{tikzcd}
Y \ar{r}{f_{\widehat X}}\ar{d}{\sigma} & \widehat{X} \ar{d}{\pi} \\
\Omega \ar{r}{p} & \X
\end{tikzcd}\caption{}
\label{fig:D}
\end{figure}

Remark that $\Gamma$ acts naturally on the product $\Omega \times_{X} \widehat{X}$, by having its natural action on the first factor, and leaving the second one invariant. Hence it also acts on $Y$.  
\medskip

Let $U \subset \widehat{X}$ be a sufficiently small neighborhood of the generic point of $E_i$. The map $f_{\widehat X}: f_{\widehat X}^{-1} (U \setminus E_i) \longrightarrow U \setminus E_i$ is an étale cover. This map induces a cyclic cover when restricted to any of the connected components of its source: the Galois group of this cover is isomorphic to $\quotientd{\mathbb Z}{m_i \mathbb Z}$ for some $m_i \in \mathbb N^\ast \cup \{\infty\}$ (with $\infty \cdot \mathbb Z = 0$). 
\medskip

Then, Definition \ref{defcovering} gives us the following natural orbifold structure on $\widehat{X}$.

\begin{defi} 
We let $\Delta_{\widehat{X}}$ be the covering divisor associated to the data $(\widehat{X}, \widehat{X} \setminus E, f_{\widehat{X}} |_{\widehat{X} \setminus E})$ by means of Definition \ref{defcovering}. By the previous discussion, it is equal to the $\mathbb Q$-divisor with simple normal crossing support $\sum_i (1 - \frac{1}{m_i}) E_i$, where the $m_i$ are defined as above. 
\end{defi}

One way to think about this orbifold structure is provided by the following formula which is direct consequence of the definition above. 
\begin{lem}
\label{formula}
With the notation above, one has
$$K_Y=f_{\widehat X}^*(K_{\widehat X}+\Delta_{\widehat X}).$$
\end{lem}

\begin{rem}
A similar, but coarser way of forming an orbifold pair $(\widehat{X}, \Delta)$ is used in \cite{crt17}. In that article, each component $E_i$ is endowed with the multiplicity $\infty$ if $\pi(E_i) \subset \X \setminus X$, and with the multiplicity $\mu_i = |S_{\pi(E_i)} |$ otherwise (where $S_{\pi(E_i)}$ is the isotropy group of the generic point of $\pi(E_i)$).
With our convention, if $\pi(E_i) \cap X \neq \emptyset$, the Galois group $\quotientd{\mathbb Z}{m_i \mathbb Z}$ identifies with a subgroup of the stabilizer of any inverse image of $\pi(E_i)$ in $\Omega$, \textsl{i.e.} it is a subgroup of the isotropy group $S_{\pi(E_i)}$. Consequently, we have $m_i \leq \mu_i$. \end{rem}

Clearly, a component $E_i$ such that $\pi(E_i) \cap X \neq \emptyset $ satisfies $m_i < \infty$. Conversely, one can prove the following:

\begin{lem}
\label{boundary}
Assume that $\Omega$  is a bounded domain satisfying the following property: for all $z\in \partial \Omega$, there exist a neighborhood  $ U_z \subset \mathbb C^n$ of $z$ and a psh function $\varphi_z$ on $U_z$ such that $\varphi_z^{-1}((-\infty,0))= \Omega\cap U_z$. 

\noindent
Then, any component $E_i$ such that $\pi(E_i) \cap X = \emptyset $ will satisfy $ m_i = \infty.$ 
\end{lem}

\begin{proof}
Note that by upper semicontinuity of $\varphi_z$, one has ${\varphi_{z}}|_{\partial \Omega \cap U}\equiv 0$. Let $E_i\subset \widehat X$ be a divisor with finite multiplicity and let us consider the étale cover $q: q^{-1} (U \setminus E_i) \longrightarrow U \setminus E_i\simeq \Delta^*\times \Delta^{n-1}$ as above. Let $V^{\circ}\subset Y$ be a connected component of $q^{-1} (U \setminus E_i)$, so that $q|_{V^{\circ}}$ can be compactified as a (surjective) ramified finite cover $\overline q:V\to U\simeq \Delta^n$ of order $m_i$ where $V$ is some smooth manifold containing $V^{\circ}$ as a Zariski open subset. In particular, one has 
\begin{equation}
\label{pi}
(\overline q \circ \pi )(V\setminus V^{\circ})= \pi(E_i\cap U).
\end{equation}
 As $\Omega\subset \mathbb C^n$ is bounded, the map $f^{\circ}:=\sigma|_{V^{\circ}}:V^{\circ}\to \Omega$ extends to a holomorphic map $f:V\to \overline \Omega$. We claim that 
 \begin{equation}
 \label{image}
 \mathrm{Im}(f)\subset \Omega
 \end{equation}
  from which the Lemma follows. Indeed one would then have $\overline q \circ \pi = p\circ f$ on $V$ by density of $V^{\circ}$ in $V$ and therefore one would get $\pi(E_i\cap U) \subset \mathrm{Im}(p) \subset X$ given \eqref{pi}. 
  
  \noindent
  We now prove \eqref{image} arguing by contradiction. Suppose that there exists $v\in V$ such that $f(v)\in \partial \Omega$. Let $z:=f(v)$ and let $(U_z,\varphi_z)$ be provided by our assumption on $\Omega$. There exists a small neighborhood $W$ of $v\in V$ such that $f(W)\subset U_z$. Then, the psh function $\varphi_z \circ f|_W$ is non-negative and attains its maximum $0$ at the interior point $v\in W$. By the maximum principle, $\varphi_z \circ f|_W$ is constant, identically equal to $0$. This is in contradiction with the fact that $(\varphi_z\circ f|_W)(W\cap V^\circ) \subset \varphi_z(\Omega\cap U_z) \subset (-\infty,0)$. 
\end{proof}

\begin{rem}
Lemma~\ref{boundary} fails for a general bounded domain. Indeed, let $\pi:\widehat X \to X'\simeq \Omega'/\Gamma$ be a resolution of a singular compact quotient $X'$ of some bounded domain $\Omega'$ and let $E\subset \widehat X$ be an irreducible, $\pi$-exceptional divisor. Then, define $\Omega:=\Omega'\setminus p^{-1}(\pi(E)), X:=X'\setminus \pi(E) $ so that $X\simeq \Omega/\Gamma$ is naturally compactified by $X'$. Then, the multiplicity of $E$ associated to $\pi:\widehat X\to X'$ is finite and yet $\pi(E)\cap X=\emptyset$.
\end{rem}

\subsection{Relative orbifold construction} 
\label{functconst} 
The previous construction has a relative variant, which uses Definition~\ref{defcovering} to construct a particular model $(\widehat{V}, \Delta_{\widehat{V}})$, once we are given a generically immersive map $q :V \longrightarrow \X$ and a resolution $\widehat{X} \longrightarrow \overline{X}$. 

Suppose here that $V$ is an $m$-dimensional complex manifold, and that the generically immersive map $q$ is such that $q(V) \not\subset  D\cup \mathrm{Sing}(p)$. Then, we will construct $\widehat{V}$ as follows. \medskip

Let $V' \overset{j} \longrightarrow \widehat{X}$ be the component of the fiber product $V \times_X \widehat{X}$ that dominates $V$. Let $\widehat{V} \longrightarrow V'$ be a  resolution of singularities; it induces a birational map $\sigma: \widehat V \to V$. Let us denote by $F\subset\widehat V$ the non-étale locus of $\sigma$, and define $F^{(1)}\subset F$ to be the union of all irreducible components of $F$ with codimension one. We also introduce the fiber product $\widetilde{V} = V \times_X \Omega$. Note that this complex space may have infinitely many connected components, all isomorphic under the action of $\Gamma$.

Finally, we let $T_{\widehat V}$ be the union of all irreducible components of the product $\widetilde{V} \times_V \widehat{V}$ dominating $\widehat V$, and we denote by $Z$ be the normalization of $T_{\widehat V}$. All these operations lead to the diagram showed in Figure~\ref{fig:MD}.
 \begin{figure}[h]
 \centering
\begin{tikzcd}[row sep=tiny, column sep = small]
 & \widehat V \arrow{dd}{\sigma} \arrow{rr}{r} & &\widehat X  \arrow{dd}{\pi}\\
Z \arrow{ur}{f_{\widehat V}} \arrow{dd}{\rho}  & & & \\
 & V \arrow{rr}{q} & & \overline{X} \\
\widetilde{V} \arrow{ur} \arrow{rr}{\iota} & & \Omega\arrow[ur,"p"'] &
\end{tikzcd}
\caption{}
\label{fig:MD}
\end{figure}

Let $\Gamma_V \subset \Gamma$ be the stabilizer of $\iota(\widetilde{V}) \subset \Omega$. Then $\Gamma_V$ acts on $V \times_X \Omega$, by its natural action on the second factor, and by the trivial action on the first. Thus, it induces a \emph{natural action on $\widetilde{V}$, making $\iota$ a $\Gamma_V$-equivariant map}.

Under these conditions, the group $\Gamma_V$ has a natural action on the fiber product $\widetilde{V} \times_V \widehat{V}$, by operating on the first factor, and leaving the second one invariant. This action leaves $T_{\widehat{V}}$ invariant and, therefore, it induces a natural action on $Z$. Again, $Z$ may have more than one connected component in general, all equivalent under the action of $\Gamma$.
\medskip

By construction, the map $f_{\widehat V}$ is étale over $\widehat V \setminus F$. By purity of the branch locus, it is actually étale over $\widehat V \setminus F^{(1)}$. Therefore, one can apply Definition~\ref{defcovering}, and endow each component $F_i \subset F^{(1)}$ with a natural multiplicity $n_i \in \mathbb N^\ast \cup \{\infty\}$. 
\medskip

\begin{defi}
We let $\Delta_{\widehat{V}}$ be the covering divisor that Definition \ref{defcovering} associates to the data $(\widehat{V}, \; \widehat{V} \setminus F, \; f_{\widehat{V}} |_{\widehat{V} \setminus F})$.  We have $\Delta_{\widehat{V}} = \sum_i (1 - \frac{1}{n_i}) F_i$, for some $n_i \in \mathbb N^\ast \cup \{ \infty \}$.
\end{defi}

\noindent
Similarly to Lemma~\ref{formula}, one has
\begin{lem}
\label{formula2}
With the notation above, one has
$$K_Z=f_{\widehat V}^*(K_{\widehat V}+\Delta_{\widehat V}).$$
\end{lem}

\subsection{The comparison result}
Our next goal is to relate the orbifold multiplicities given by the divisor $\Delta_{\widehat{V}}$ with the ones inherited from the pair $(\widehat{X}, \Delta_{\widehat{X}})$: this will be the content of Proposition \ref{propeffective}. Before this, we need a lemma.

\begin{lem}
\label{cartesian}
The variety $Z$ is naturally isomorphic to the normalization of the union of the components of $\widehat V \times_{\widehat X} Y$ dominating $\widehat V$. 
\end{lem}

\begin{proof}
Note that the associativity of fiber products yields 
$$\widehat{V} \times_V \widetilde{V}= \widehat V \times_V  (V \times_{\overline X} \Omega) \simeq \widehat V \times_{\overline X} \Omega \simeq  \widehat V \times_{\widehat X}( \widehat X   \times_{\overline X} \Omega). $$
From this, we get that $T_{\widehat V}$, the disjoint union of the components of $\widehat{V} \times_V \widetilde{V}$ dominating $\widehat V$, identifies with the disjoint union of the components of $\widehat V \times_{\widehat X} T_{\widehat X}$ dominating $\widehat V$.
Now, the universal property of the normalization functor ${\bullet}^\nu$ allows us to complete the square as follows
$$
\begin{tikzcd}
 (\widehat V \times_{\widehat X} T_{\widehat X}^{\nu})^{\nu} \arrow{d} \arrow[dashed]{r} & (\widehat V \times_{\widehat X} T_{\widehat X})^{\nu} \arrow{d} \\
 \widehat V \times_{\widehat X} T_{\widehat X}^{\nu} \arrow{r} & \widehat V \times_{\widehat X} T_{\widehat X}
\end{tikzcd}
$$
Now, the dotted arrow represents a finite bimeromorphic map between two normal reduced complex analytic spaces hence it is an isomorphism. As $Y=T_{\widehat X}^{\nu}$, the normalization of the disjoint union of components of $\widehat V \times_{\widehat X}  Y$ dominating  $\widehat V$ is the same thing as the disjoint union of components of $ (\widehat V\times_{\widehat X} T_{\widehat X})^{\nu} $ dominating  $\widehat V$. By what was said previously, this is nothing but saying that $T_{\widehat V}^{\nu}=Z$. 
\end{proof}

\noindent
The natural orbifold structures on $\widehat{V}$ and $\widehat{X}$ are now comparable in the following manner.

\begin{prop} 
\label{propeffective} 
With the notation above, one has 
 $$\Delta_{\widehat{V}} \leq r^\ast \Delta_{\widehat X}, $$ \textsl{i.e.} the difference $r^\ast \Delta_{\widehat X}-\Delta_{\widehat{V}}$ between these two $\mathbb Q$-divisors is effective.
\end{prop}

\begin{proof} 
We have the following commutative diagram:
$$
\begin{tikzcd}[row sep=tiny, column sep = small]
 & \widehat V \arrow[dotted]{dd} \arrow{rr}{r} & &\widehat X  \arrow{dd}\\
Z \arrow{ur}{f_{\widehat V}} \arrow{dd}  \arrow{rr}[near end]{s} & & Y \arrow{dd} \arrow{ur}[near start]{f_{\widehat X}} & \\
 & V \arrow[dotted]{rr} & & \overline{X}. \\
\widetilde{V} \arrow{ur} \arrow{rr} & & \Omega\arrow{ur} &
\end{tikzcd}
$$

We claim that over $\widehat V \setminus F$, we have $Z\simeq \widehat V \times_{\widehat X} Y$. Given Lemma~\ref{cartesian}, it is sufficient to prove that $Z|_{f_{\widehat V}^{-1}(\widehat V \setminus F)}$ is smooth and that each of its connected components dominates $\widehat V$. As $\sigma$ is an isomorphism over $\widehat V \setminus F$, it suffices to check those properties for $\widetilde V \to V$ over that same set but this is then straightforward.

Let $\eta$ be a general point of a component $F_i$ of $F$ such that $2 \le n_i \le+\infty$ and let $U_{\widehat V} \subset \widehat{V}$ be a small neighborhood of $\eta$ on which $F_i$ admits the equation $v_i = 0$. Let $W_{\widehat V}$ be a connected component of $f_{\widehat V}^{-1}(U_{\widehat V} \setminus F_i)$. By Definition \ref{defcovering}, the map $W_{\widehat V} \overset{f_{\widehat V}}{\longrightarrow} U_{\widehat V} \setminus F_i$ is an \'etale cover, with Galois group $G_i = \quotientd{\mathbb Z}{n_i \mathbb Z}$. As $n_i\ge 2$, $F_i$ sits above $\mathrm{Branch}(p)$ as otherwise, $f_{\widehat V}: f_{\widehat V} \to U_{\widehat V}\simeq \Delta^m$ is étale hence an isomorphism. 

Set $\zeta = r(\eta)$, and let $U_{\widehat X}$ be a small neighborhood of $\zeta$ containing $r(U_{\widehat V})$. Denote also by $W_{\widehat X}$ the connected component of $f_{\widehat X}^{-1}( U_{\widehat X} \setminus E)$ containing $s(W_{\widehat V})$. We get a map of \'{e}tale covers
\begin{equation}
\begin{tikzcd}
 \label{diagcovers}
W_{\widehat V} \arrow{r}{s} \arrow[d,"f_{\widehat V}"'] & W_{\widehat X} \arrow{d}{f_{\widehat X}} \\
U_{\widehat V} \setminus F_i \arrow{r}{r} & U_{\widehat X} \setminus E 
\end{tikzcd}
\end{equation}
and we know from our observation at the beginning of the proof that the diagram 
\begin{equation}
\begin{tikzcd}
 \label{diagcovers2}
s^{-1}(W_{\widehat X}) \arrow{r}{s} \arrow[d,"f_{\widehat V}"'] & W_{\widehat X} \arrow{d}{f_{\widehat X}} \\
U_{\widehat V} \setminus F_i \arrow{r}{r} & U_{\widehat X} \setminus E 
\end{tikzcd}
\end{equation}
is a fiber product. Therefore, we have
\begin{align*}
n_i&=\mathrm{deg}\left(W_{\widehat V} \overset{f_{\widehat V}}{\longrightarrow} U_{\widehat V} \setminus F_i \right)\\
&\le  \#(\mbox{fiber of } \,s^{-1}(W_{\widehat X})  \overset{f_{\widehat V}}{\longrightarrow} U_{\widehat V} \setminus F_i )\\
& =  \mathrm{deg}\left(W_{\widehat X}  \overset{f_{\widehat X}}{\longrightarrow} U_{\widehat V} \setminus F_i \right)
\end{align*}
If $H$ is the Galois group of $f_{\widehat X}:W_{\widehat X} \longrightarrow U_{\widehat X} \setminus E$, we have $n_i \leq |H|$. Let $(E_j)_{j \in J}$ be the components of $E$ passing through $\zeta$. Since $m_j$ is the order of the element of $H$ associated to the meridian loop around $E_j$, the proof of \cite[Theorem 2.23]{kol07} shows that $H$ is an abelian group satisfying $$|H| \leq \prod_{j\in J} m_j.$$ 

Given $j\in J$, let us introduce $z_j$ a local equation for $E_j$.
Since $v_i$ divides each $r^\ast z_j$ in $\mathcal O_\eta$, to finish the proof, it suffices to show that $1 - \frac{1}{n_i} \leq \sum_{j \in J} \left( 1 - \frac{1}{m_j} \right)$. But this is now an easy consequence of the inequality $n_i \leq \prod_{j \in J} m_j$ obtained previously. 
\end{proof}

\section{A criterion for hyperbolicity}
\label{criterion}

The main goal of this section is to present a hyperbolicity result for the complex space $\overline{X}$, provided that the manifold $\Omega$  in the general Assumption \ref{generalassumption} is actually a bounded domain. The section is organized as follows: 

\noindent
$\bullet$ In Section~\ref{sec:Bergman}, we give some complements on the Bergman metric and how to compute its curvature, cf. \eqref{eqcurvature}.  

\noindent
$\bullet$ In Section~\ref{sec:ineq}, we gather a few results allowing us to estimate the curvature of the Bergman kernel on $\Omega$, close to the classical comparison theorems between the Bergman, Carath\'eodory and Kobayashi metrics (see \cite{hah78, kob98}). The main result of that section is Proposition~\ref{lemcontcurvfin}. 

\noindent
$\bullet$ In Section~\ref{sectsingmetorb}, we build on the previous section to construct a singular Hermitian metric on a modification of a weakly pseudoconvex Kähler manifold
 $V$ admitting a generic immersive map to $\overline X$. The main result of the section is a curvature inequality for  that metric, cf. Theorem~\ref{eqcontrolcurv}.
 
\noindent
$\bullet$ In Section~\ref{sec:criterion}, we exploit the previous results to state and prove a hyperbolicity criterion for $\overline X$, cf. Theorem~\ref{thmsingmetric2}.\\

 Throughout the rest of this section, we assume that $\Omega \subset \mathbb C^n$ is a bounded domain.

\subsection{Computation of the curvature of the Bergman metric}
\label{sec:Bergman}
Let $Y$ be a complex manifold of dimension $n$. Let us give some complements on the discussion of Section \ref{subsectbergmanmetric}, and briefly recall how to compute the curvature $i\Theta(h_Y)$ of the Bergman metric $h_Y$ on $K_Y$ (when it is defined).

Let $\mathcal H_Y = \left\{ \sigma \in H^0(Y,K_Y) \; | \; \int_Y i^{n^2} \sigma \wedge \overline{\sigma}<+\infty \right\}$ be the Hilbert space of holomorphic square integrable $n$-forms on $Y$. If $e$ is some local trivialization of $K_Y$, the norm of $e$ for the metric $h_Y$ has the following value at a point $x \in Y$: 
$$
\norm{e}_{h_Y, x} = \frac{1}{\norm{\mathrm{ev_x}}_{\mathcal H_Y^\ast}},
$$
where $\mathrm{ev_x}: \mathcal H_Y \longrightarrow \mathbb C$ is the evaluation form which to $\sigma \in \mathcal H_Y$ associates $\lambda$ such that $\sigma_x = \lambda e_x$, and $\norm{ \cdot }_{\mathcal H_Y^\ast}$ is the natural dual norm on $\mathcal H_Y^\ast$. Thus, the Bergman metric at $x$ is well defined provided there exists $\sigma \in \mathcal H_Y$ such that $\sigma_x \neq 0$.
\medskip

Consider now a point $x \in Y$ such that $\norm{ \cdot}_{h_Y, x}$ is defined. By definition of $h_Y$, there exists a section $e \in \mathcal H_Y$ such that $\norm{e}_{\mathcal H_Y} = 1$ and $\norm{e_x}_{h_Y, x} = 1$. Now, if $v \in T_{Y, x}$, the curvature of $h_Y$ in the direction $v$ can be computed by the following formula (see \cite[Proposition 4.10.10]{kob98}). 
\begin{equation} \label{eqcurvature}
i\Theta(h_Y) (v, \overline{v}) = \max_{\sigma} \abs{df(v)}^2,
\end{equation}
where $\sigma \in \mathcal H_Y$, with $\norm{\sigma}_{\mathcal H_Y} = 1$ and $\sigma(x) = 0$, and $f\in \mathfrak m_{Y,x} \subset \mathcal O_{Y,x}$ is such that locally around $x$ one has $\sigma= f e$ (recall that $e_x\ne 0$, so that $e$ gives a local holomorphic frame around $x$).


\subsection{Curvature inequalities on subvarieties}
\label{sec:ineq}

We will now use the previous description of the curvature of $h_Y$ to state a comparison result between the curvature of the Bergman metric of a bounded domain and that of a bounded \emph{symmetric} domain included in it. We will then use this result to obtain a curvature estimate for the subvarieties of $\Omega$. 

Let $\mathcal D$ be a bounded \emph{symmetric} domain of dimension $n$, centered at $0 \in \mathbb C^n$, with coordinates $(t_1, ..., t_n)$. Since $\mathcal D$ is $S^1$-invariant, we see immediately that two polynomials $t^\alpha = t_1^{\alpha_1} ... t_n^{\alpha_n}$ ($\alpha = (\alpha_1, ..., \alpha_n)$) and $t^\beta = t_1^{\beta_1} ... t_n^{\beta_n}$ ($\beta = (\beta_1, ..., \beta_n)$) are orthogonal for the standard scalar product, whenever $\alpha \neq \beta$. After renormalizing the family $(t^\alpha \, dt_1 \wedge ... \wedge dt_n)_{\alpha \in \mathbb N^n}$, we get a unitary basis $(e_i)_{i \in \mathbb N}$ of $\mathcal H_{\mathcal D}$, of the form $e_i = f_i dt_1 \wedge ... \wedge dt_n$, with
$$
f_0 = \frac{1}{\mathrm{Vol}(\mathcal D)^{\frac{1}{2}}}, \; \; f_1 = \frac{1}{a_1} t_1, \;\; ..., \;\; f_n = \frac{1}{a_n} t_n,
$$
where $a_i^2 = \int_{\mathcal D} |t_i|^2 \mathrm{dVol}$, all other $f_i$ being polynomials in $t$ with vanishing $1$-jet at $0$.

This implies that
$$
\norm{dt_1 \wedge ... \wedge dt_n }^2_{h_{\mathcal D, 0}} = \frac{1}{\sum_{i \in \mathbb N} |f_i(0)|^2} = \mathrm{Vol}(\mathcal D),
$$

Let $v \in T_{\mathcal D, 0}$. Taking $e = e_0$, the equality case in Cauchy-Schwarz inequality shows that the maximum in \eqref{eqcurvature} is attained for $\sigma = f e$ with 
\begin{equation}
\label{f}
f = \frac{\sum_i \overline{v_i} t_i/a_i^2 }{(\sum_i |v_i|^2/a_i^2)^{1/2}}\cdotp \mathrm{Vol}(\mathcal D)^{1/2}.
\end{equation}
This yields, by \eqref{eqcurvature}:
\begin{equation} \label{curvboundeddomain}
i \Theta(h_{\mathcal D}) (v, \overline{v}) = \left(\sum_{i} \frac{|v_i|^2}{a_i^2} \right) \cdotp \mathrm{Vol}(\mathcal D)
\end{equation}

We are now ready to state our first comparison result.

\begin{lem} \label{lemcontrol1}
Let $x \in \Omega$. Let $j: \mathcal D \hookrightarrow \Omega$ be an open embedding, such that $j(0) = x$. Then, we have
$$
j^\ast \left( i \Theta(h_{\Omega})_x \right) \leq \frac{\mathrm{Vol}(\Omega)}{\mathrm{Vol}(\mathcal D)} \frac{1}{|\mathrm{Jac} (j)(0)|^2}  i \Theta(h_{\mathcal D})_0.
$$
\end{lem}
\begin{proof}
Let $w \in T_{\mathcal D, 0}$, and let $v = j_\ast(w)$. We are going to show that the inequality holds when applied to $w$. 

We first gather a few objects allowing us to compute the left hand side. Accordingly to \eqref{eqcurvature}, we let $e, \sigma \in \mathcal H_{\Omega}$ be such that $\norm{e}_{\mathcal H_{\Omega}} = \norm{\sigma}_{\mathcal H_{\Omega}} = 1$ and $\norm{e_x}_{h_{\Omega, x}} = 1$, $\sigma(x) = 0$, and we finally require that $i \Theta(h_{\Omega}) (v, \overline{v}) = |df(v)|^2$, where $\sigma \overset{loc}{=} f e$ near $x$. Writing $\sigma = g \,dz_1\wedge \ldots \wedge dz_n$, and $e = g_0 \, dz_1 \wedge \ldots \wedge dz_n$, we get the alternate expression
\begin{equation} \label{eqexpraltcurv}
i \Theta(h_{\Omega})_x (v, \overline{v}) = \frac{|dg(v)|^2}{|g_0(x)|^2}.
\end{equation}
Remark that since $\int_{\Omega} \frac{d\mathrm{vol}_{\mathbb C^n}}{\mathrm{Vol}(\Omega)} = 1$, we must have 
\begin{equation} \label{eqineqg0}
|g_0(x)|^2 \geq \frac{1}{\mathrm{Vol(\Omega)}}
\end{equation}
since $g_0$ realizes the supremum of the evaluation function at $x$ on $B(0,1) \subset \mathcal H_Y$.
\medskip

To compute the right hand side, remark first that
$$
j^\ast \sigma = \mathrm{Vol}(\mathcal D)^{1/2} \, (g \circ j) \, \mathrm{Jac}(j) \, \left[ \frac{dt_1 \wedge ... \wedge dt_n}{\mathrm{Vol}(\mathcal D)^{1/2}} \right].
$$
Denote by $e_{\mathcal D}$ the term between brackets. We have seen previously that $\norm{e_{\mathcal D}}_{\mathcal H_{\mathcal D}} = \norm{e_{\mathcal D, 0}}_{h_{\mathcal D}} = 1$. Moreover, since $j$ is an open immersion, we have $\norm{j^\ast \sigma}_{\mathcal H_{\mathcal D}} \leq 1$.

These two facts allow us to use \eqref{eqcurvature} to bound  the curvature of $h_{\mathcal D}$ from below, writing $j^\ast \sigma = f_{\mathcal D} e_{\mathcal D}$, with $f_{\mathcal D} = \mathrm{Vol}(\mathcal D)^{1/2} \, (g \circ j) \, \mathrm{Jac}(j)$. We get 
\begin{align*}
i \Theta(h_{\mathcal D})(w, \overline{w}) & \geq |df_{\mathcal D} (w)|^2 \\
& \geq | d( \mathrm{Vol}(\mathcal D)^{1/2} \;  (g \circ j) \; \mathrm{Jac}(j)) \cdot w|^2 \\
&  = \mathrm {Vol}(\mathcal D) \; |\mathrm{Jac}(j)(0)| \; | d(g \circ j)(w))|^2  \\
& \geq \frac{\mathrm{Vol(\mathcal D)}}{\mathrm{Vol}(\Omega)} \; | \mathrm{Jac}(j)(0) |^2 \; \frac{|dg(v)|^2}{|g_0(x)|^2}
\end{align*}
where at the second line, we used the fact that $g(x) = 0$, and at the last line, we used \eqref{eqineqg0}. The last equation, combined with \eqref{eqexpraltcurv}, allows us to end the proof.
\end{proof}

\begin{rem} In particular, if $r = d(x, \partial \Omega)$, we can apply the previous lemma to the open embedding of the ball $B(x, r) \hookrightarrow \Omega$, with $j = {\mathrm{Id}}$. This gives, for any $v \in T_{\Omega, x}$:
\begin{align*}
i \Theta(h_{\Omega})_x (v, \overline{v}) & \leq \frac{\mathrm{Vol}(\Omega)}{\mathrm{Vol}(B(x, r))} i \Theta(h_{B(x, r)}) (v, \overline{v}) \\
 & = \frac{\mathrm{Vol}(\Omega)}{\mathrm{Vol}(B(x,r))^2}\cdotp \frac{n+1}{r^2} \norm{v}^2_{\mathbb C^n}
\end{align*}
using \eqref{curvboundeddomain} and the fact that for $\mathcal D = B(x, r)$, we have $a_i^2 = \mathrm{Vol}(B(x,r))\cdotp \frac{r^2}{n+1}$ for any $i$. 
\end{rem}

The next lemma will be used later on to estimate the curvature of the Bergman metric on subvarieties of $\Omega$.

\begin{lem} \label{lemcontrolcurvature1} Assume that $0 \in \mathcal D \subset \Omega$, and that $\mathcal D$ is centered at $0$. Let $Y$ be a complex manifold, and let $Y \overset{q}{\longrightarrow} \Omega$ be a generically immersive holomorphic map passing through $0$. Choose $y \in q^{-1}(0)$, and let $d_{\Omega,0} = \max_{z \in \partial \Omega} d(0, z)$.

Suppose that $h_Y$ is defined at $y$. Then, we have:
$$
i \Theta(h_Y)_y \geq \frac{ \min_i a_i^2}{ \mathrm{Vol}(\Omega)\, d_{\Omega, 0}^2}  q^*(i \Theta(h_\Omega)_0).
$$
\end{lem}

\begin{proof}
Fix a vector $v \in T_{Y, y}$, and let $w = q_\ast (v)$. We want to show that the inequality holds when applied to $v$. We may suppose that $w \neq 0$, the inequality being trivial otherwise.

Since $h_Y$ is defined at $y$, there exists $\eta \in \mathcal H_Y$ such that $\norm{\eta}_{\mathcal H_Y} = 1$, and $\norm{\eta_y}_{h_Y, y} = 1$. Besides, by \eqref{f} and \eqref{curvboundeddomain}, we have
$$
i \Theta(h_{\mathcal D}) (w, \overline{w}) = |df(w)|^2,
$$
with $f(z) = \frac{\sum_i \overline{w_i} z_i/a_i^2 }{(\sum_i |w_i|^2/a_i^2)^{1/2}}\cdotp \mathrm{Vol}(\mathcal D)^{1/2}$. Note that by Cauchy-Schwarz inequality, we get the following upper bound: 
$$
|f(z)|^2 \leq  \left( \sum_{i} \frac{|z_i|^2}{a_i^2} \right) \mathrm{Vol}(\mathcal D) \leq \frac{d_{\Omega, 0}^2}{\min_i a_i^2} \mathrm{Vol}(\mathcal D).
$$
Define 
$$
\begin{aligned}
g\colon\, &\mathbb C^n\longrightarrow \mathbb C \\
& z\longmapsto \frac{\min_i a_i }{\mathrm{Vol}(\mathcal D)^{1/2} {d_{\Omega, 0}}} f(z).
\end{aligned}
$$ 
Then, we have $\sup_Y |g \circ q | \leq 1$, so $\sigma = (g \circ q) \eta \in \mathcal H_Y$, and by \eqref{eqcurvature}, we get
\begin{align*}
i \Theta(h_Y)(v, \overline{v}) & \geq |d(g \circ q)(v)|^2 \\
			       & \geq \frac{\min_i a_i^2}{\mathrm{Vol} (\mathcal D) d_{\Omega, 0}^2} |df(w)|^2 
\end{align*}
This shows that $i \Theta(h_Y)_y \geq \frac{\min_i a_i^2}{\mathrm{Vol}(\mathcal D)d_{\Omega, 0}^2} i q^\ast \Theta(h_{\mathcal D})_0$. Using Lemma \ref{lemcontrol1} with $j = \mathrm{Id}_{\mathbb C^n}$, we see that $i \Theta(h_{\mathcal D})_0 \geq \frac{\mathrm{Vol}(\mathcal D)}{\mathrm{Vol}(\Omega)} i \Theta(h_\Omega)_0$. This ends the proof.
\end{proof}


We now make the following regularity assumption on the bounded domain $\Omega$.

\begin{assumption}  \label{funddomainassumption}
The manifold $\Omega$ is a bounded domain admitting a \emph{cocompact} discrete subgroup $\Gamma_{0} \subset \mathrm{Aut}(\Omega)$. Let $\mathfrak{U}_0 \subset \Omega$ be a compact fundamental domain for $\Gamma_{0}$, and let $r_0 = d(\mathfrak{U}_0, \partial \Omega)$, $d_0 = \max_{x \in \mathfrak{U}_0, z \in \partial \Omega} d(x, z)$. 
\end{assumption}

Under this assumption, we can obtain a uniform bound in Lemma \ref{lemcontrolcurvature1}, in terms of some constant depending on $\Gamma_0$. 

\begin{defi}
Under the hypothesis of Assumption \ref{funddomainassumption}, we introduce the following constant
$$
C_0 = \frac{1}{d_0^2 \, \mathrm{Vol}(\Omega)} \sup_{x \in \mathcal D \subset \Omega} \left( \mathrm{min}_i \; a_i^2  \right),
$$
where $x$ runs among the points of $\mathfrak{U}_0$, $\mathcal D$ runs among the bounded symmetric domains centered at $x$ and included in $\Omega$, and the $a_i$ are the constants associated to $\mathcal D$.
\end{defi}

Remark that since $d(\mathfrak{U}_0, \partial \Omega) = r_0 > 0$, we can always take $\mathcal D = B(x, r_0)$ in the previous definition. Then, an easy computation shows that
$$
C_0 \geq \frac{1}{n+1} \frac{\mathrm{Vol}(B(0, r_0))}{\mathrm{Vol}(\Omega)} \frac{r_0^2}{d_0^2}.
$$
Note that we also have the trivial upper bound $d_0\le \mathrm{diam}(\Omega)$. 
\begin{prop} 
\label{lemcontcurvfin}
Let $Y$ be a complex manifold, and let $Y \overset{q}{\longrightarrow} \Omega$ be a generically immersive holomorphic map. Suppose that $h_Y$ is well defined at a generic point of $Y$. Then, we have
$$
i \Theta(h_Y) \geq C_0\, q^\ast ( i\Theta(h_\Omega))
$$
in the sense of currents.
\end{prop}
\begin{proof}
Since the right hand side is continuous on $Y$, it suffices to prove the inequality at any point $y$ where $h_Y$ is non-degenerate.
The right hand side being invariant under the action of $\Gamma_0 \subset \mathrm{Aut}(\Omega)$, we can let this lattice act on $\Omega$ and assume that $x = q(y) \in \mathfrak{U_0}$. One can now apply Lemma \ref{lemcontrolcurvature1} to any bounded symmetric domain included in $\Omega$ and centered at $x$; this concludes the proof of the proposition. 
\end{proof}
\medskip

\subsection{A uniform curvature inequality}
 \label{sectsingmetorb}

We keep working under the Assumption \ref{funddomainassumption} on $\Omega$, and we keep using the symbols $\Gamma_{\widetilde V}$, $\mathfrak{U}_0$ and $C_0$, with the same meaning as before, cf. Section~\ref{orbifold} and Figure~\ref{fig:MD} that we reproduce below as Figure~\ref{fig:MD2} for the reader's convenience. 
 \begin{figure}[h]
 \centering
\begin{tikzcd}[row sep=tiny, column sep = small]
 & & & \widehat V  \arrow{dd}{\sigma} \arrow{rr}{r} & &\widehat X  \arrow{dd}{\pi}\\
\widehat Z   \arrow{rr}{\sigma_{Z}}  & & Z \arrow{ur}{f_{\widehat V}} \arrow{dd}{\rho}  & & & \\
& & & V \arrow{rr}{q} & & \overline{X} \\
& &\widetilde{V} \arrow{ur} \arrow{rr}{\iota} & & \Omega\arrow[ur,"p"'] &
\end{tikzcd}
\caption{}
\label{fig:MD2}
\end{figure}
In particular, the map $q:V\to \overline X$ is a generically immersive map and $\sigma : \widehat V \to V$ is a suitable modification making the diagram commutative. In the following, we will assume that $V$ is a weakly pseudoconvex $m$-dimensional Kähler manifold. This means that there exists a smooth plurisubharmonic exhaustion function $\psi: V \longrightarrow \mathbb R$. 
We want to show that the $\mathbb Q$-line bundle $K_{\widehat{V}} + \Delta_{\widehat{V}}$ admits a natural singular metric with positive curvature. We first make the following remark, which follows directly from Proposition~\ref{lemcontcurvfin}.

\begin{lem} 
\label{lem:ineqcurv}
Suppose that the Bergman metric $h_{Z_{\rm reg}}$ is well defined at a generic point of $Z_{\rm reg}$. Then the $\Gamma_V$-invariant metric $h_{Z_{\rm reg}}$ on $K_{Z_{\rm reg}}$ has positive curvature, satisfying
$$
i \Theta (h_{Z_{\rm reg}}) \geq C_0 \;  j^\ast (i\Theta(h_{\Omega})),
$$ 
where $j:=(i\circ \rho)|_{Z_{\rm reg}}: Z_{\rm reg} \longrightarrow \Omega$ is the natural map.
\end{lem}

The next lemma relies on an adaptation to the non-compact case of some classical arguments in K\"ahler geometry (see \textsl{e.g.} \cite{DP}). Let $\widehat{Z} \overset{\sigma_Z}{\longrightarrow} Z$ be some resolution of singularities to be fixed later and let $f_V:\widehat Z \to V$ be defined as the composition $f_V:=\sigma \circ f_{\widehat V} \circ \sigma_{Z}$.

\begin{lem} \label{lemkahler}
Assume that $V$ admits a K\"ahler metric $\omega$, let $(V_i)_{i \in \mathbb N}$ be an exhaustive sequence of relatively compact open subsets of $V$ and set $\widehat{Z}_i := f_V^{-1} (V_i)$. Then, for an adequate choice of desingularizations $\widehat{V}$ and $\widehat{Z}$, each manifold $\widehat{Z}_i$ admits a K\"ahler metric $\omega_i$. 

\noindent
Moreover, we can choose the metrics $\omega_{i}$ so that
\begin{equation} \label{cvmetrics}
\omega_{i} \underset{i \longrightarrow + \infty} \longrightarrow f_V^* \, \omega
\end{equation}
where the convergence holds uniformly on compact subsets of $\widehat{Z}$.
\end{lem}
\begin{proof}

We may replace $\widehat{V}$ by a resolution of indeterminacies of  the bimeromorphic map $V \dashrightarrow \widehat{V}$, to suppose that $\sigma: \widehat{V} \longrightarrow V$ is obtained by a sequence of blow-ups along smooth centers. Remark that this sequence may be infinite; however, the centers project onto a locally finite family of subsets of $V$.

Let $E$ be the exceptional divisor of $\sigma$, with irreducible components $E = \sum_{k \in \mathbb N} E_k$. A classical argument allows one to find smooth $(1,1)$-forms $\theta_{E_k} \in c_1(E_k)$ with support in an arbitrarily small neighborhood of $E_k$ and a sequence of positive numbers $(a_k)$ such that the (locally finite) sum  $\theta_E = \sum_k a_k \theta_k$ defines a $(1, 1)$-form on $\widehat{V}$ which is negative definite along the fibers of $\sigma$.
 Fix now some $i \in \mathbb N$. Since $V_i$ is relatively compact in $V$, for $\epsilon_i > 0$ small enough, the closed $(1, 1)$-form $$
\omega_{\widehat{V}_i} =  \sigma^\ast \omega - \epsilon_i \, \theta_E
$$
defines a Kähler metric on $\sigma^{-1} (V_i)$.
\medskip

Now, let $\widehat{Z} \overset{\sigma_Z}{\longrightarrow} Z$ be a resolution of singularities obtained by blowing-up smooth centers, and let $\widehat{p}: \widehat{Z} \longrightarrow \widehat{V}$ be the induced map. We ask that the strict transform $F = \widehat{p}_\ast \,^{-1}(\Delta_{\widehat{V}})$ is a \emph{disjoint} union of smooth hypersurfaces, and that $F$ has simple normal crossings with the exceptional divisor $E'$ of the map $\sigma_Z$.  Using partitions of unity, we can easily construct a smooth function $\phi$ on $\widehat{Z}$ so that $i \partial \overline{\partial} \phi$ is positive in the directions transverse to the ramification divisor $F$. As before, we also let $\theta_{E'} \in c_1(E')$ be negative definite along the fibers of $\sigma_{Z}$.

With these definitions, for $\epsilon_i' > 0$ small enough, the closed $(1, 1)$-form
$$
\omega_{i} = \widehat{p} \,^{\ast} \omega_{\widehat{V}_i} + \epsilon'_i \left( i \partial \overline{\partial} \phi - \theta_{E'} \right),
$$
defines a K\"ahler metric on $\widehat{Z}_i$.

For the second requirement to be satisfied, we just need to take $\epsilon_i$ and $\epsilon_i'$ decreasing to $0$ as $i \longrightarrow + \infty$.
\end{proof}

The next proposition is an adaptation to the non-compact case of the main argument of \cite{BD18}. It is the last step towards Theorem \ref{thmsingmetric}, which is the main result of this section.

\begin{prop} 
\label{weakbergman}
Assume that $V$ is a weakly pseudoconvex K\"ahler manifold. Then, we can choose $\widehat{V}$ and $\widehat{Z}$ so that $\widehat{Z}$ is a \emph{weakly Bergman manifold}.
\end{prop}

\begin{rem}
If $V$ is \textit{compact} Kähler, then $\widetilde V$ is complete Kähler and admits a generically immersive map towards the bounded domain $\Omega$, and hence the conclusion follows directly from \cite{BD18}. 
\end{rem}
\begin{proof}

Let $\psi: V \longrightarrow \mathbb R$ be a smooth exhaustive plurisubharmonic function. For each $i \in \mathbb N$, we let $V_i = \psi^{-1}([0, i[)$, and we fix $\widehat{Z}$, $\widehat{V}$ and $(\omega_i)_{i \in \mathbb N}$ as provided by Lemma \ref{lemkahler}. We will show that $i \Theta (h_{\widehat{Z}})$ is positive definite at a generic point of $Z_{\rm reg}$.  By definition of the Bergman metric, it suffices to show that the $L^2$ holomorphic $m$-forms on $\widehat{Z}$ generate the $1$-jets at any generic point of $Z_{\rm reg}$. 
\medskip

Let $z \in Z_{\rm reg}$ be a point belonging to the regular loci of the maps $j = \iota \circ \rho$ and $\sigma \circ f_{\widehat V}$ (cf. Figure \ref{fig:MD2}).  One picks a germ $\tau$ of holomorphic $m$-form at $z$ (recall that $m = \dim Z$). Remark now that each $\widehat{Z}_i$ is weakly pseudoconvex because $V_i$ is weakly pseudoconvex and the natural maps $\widehat{Z}_i \longrightarrow V_i$ are proper. Since each $\widehat{Z}_i$ admits the K\"ahler metric $\omega_i$, this allows us to use the $L^2$-method on $\widehat{Z_i}$ with $\omega_i$ (see \cite[Theorem 6.1]{dem12}).
\medskip

To do this, we choose a cutoff function $\chi$ on $\widehat{Z}$, equal to $1$ in a neighborhood of $z$, and with compact support $L \subset \widehat{Z_{i_0}} \cap \sigma_Z^{-1}(Z_{\rm reg})$ for some $i_0 \geq 0$. Without loss of generality, one can assume that $L$ is contained in the regular locus of $ f_V :\widehat Z\to V$. Let us define
\begin{align*}
\varphi :  \, \,\Omega & \longrightarrow  \mathbb R\\
 x  &\longmapsto 2(n+1) \log |x - j(z)| + |x|^2.
\end{align*} The function $\varphi$ is psh on $\Omega$, and $\widehat \varphi :=\varphi \circ \iota \circ \rho \circ \sigma_Z$ is strictly psh at $\sigma_Z^{-1}(z)$. Note that both $\chi$ and $\varphi$ are independent of $i$.

As explained above, we can apply the $L^2$ method on $\widehat{Z}_i$ to deduce that there exists a smooth $(m, 0)$-form $f_i$ on each $\widehat{Z}_i$, satisfying
\begin{equation} \label{L2method}
i^{m^2} \int_{\widehat{Z}_i} f_i \wedge \overline{f_i} \; e^{- \widehat \varphi} \leq \int_{L} | \overline{\partial} (\chi \tau) |_{\omega_i}^2\, e^{- \widehat \varphi} dV_{\omega_i}.
\end{equation}
Thanks to \eqref{cvmetrics}, and since the metric $f_V^\ast \, \omega$ is non-degenerate on the compact set $L \subset Z_{\rm reg}$, the right hand side of the above equation is uniformly bounded by some constant $C$ for any $i \geq i_0$. Since $\varphi$ is bounded from above, this implies a uniform bound
\begin{equation} \label{eqbound}
\norm{f_i}^2_{L^2(\widehat{Z_i})} = \int_{\widehat{Z}_i} i^{m^2} f_i \wedge \overline{f_i} \; \leq \; C \, e^{\sup_\Omega \varphi}.
\end{equation}

The expression of $\varphi$ is chosen so that the bound \eqref{L2method} implies that $f_i$ has a vanishing $1$-jet at $z$. Thus, for any $i$, $\eta_i = \chi \tau - f_i$ is a holomorphic $m$-form on $\widehat{Z_i}$ with jet $(d \eta_i)_z = d \tau_z$ at $z$. Also, \eqref{eqbound} provides a uniform bound 
$$
\sup_i \norm{\eta_i}^2_{L^2(\widehat{Z_i})} \leq C'.
$$
 Thus, we can extract a sequence converging uniformly on compact subsets towards a holomorphic form $\eta$. This form $\eta$ satisfies $d \eta_z = d \tau_z $. Also, by Fatou lemma, we have $\norm{\eta}^2_{L^2(\widehat{Z})} \leq C'$. Thus, $\eta$ satisfies our requirements.
\end{proof}

We are now ready to prove the main result of this section. 

\begin{thm} 
\label{thmsingmetric} 

Let $q:V\to \overline X$ be a generically immersive map from a \emph{weakly pseudoconvex K\"ahler manifold} $V$ such that $q(V)\not \subset D \cup \mathrm{Sing}(p)$.

\noindent
Provided that Assumption~\ref{funddomainassumption} is satisfied, we can choose $\widehat{V}$ so that the $\mathbb Q$-line bundle $\mathcal O_{\widehat V}(K_{\widehat{V}} + \Delta_{\widehat{V}})$ admits a natural singular metric $g_{\widehat{V}}$ with positive curvature, satisfying
\begin{equation} \label{eqcontrolcurv}
i \Theta (g_{\widehat{V}}) \geq C_0 \;  (q\circ \sigma)^\ast i \Theta(g_{X})
\end{equation}
over $V \setminus \left(  q^{-1}(D \cup \mathrm{Sing}(p)) \right)$. 
\end{thm}

\begin{rem}
In the particular case where $V$ is a \emph{compact K\"ahler manifold}, we recover the case obtained in Theorem~\ref{newthm} : $K_{\widehat{V}} + \Delta_{\widehat{V}}$ is big by \cite{bou02}. 
\end{rem}

\begin{proof}

Let $\widehat{V}^\circ \subset \widehat V$ be the étale locus of the cover $f_{\widehat V}$ and let $F:=\widehat V \setminus \widehat{V}^\circ$. By purity of the branch locus, $F$ has pure codimension one. Let $n_i\in \mathbb N^*\cup \{\infty\}$ be the covering multiplicity attached to an irreducible component $F_i$ of $F$ and let $\Delta_{\widehat V}:=  \sum_i (1 - \frac{1}{n_i}) F_i$.

Since the Bergman metric $h_{Z_{\mathrm{\rm reg}}}$ is invariant under the group $\Gamma_V$,  it descends to define a singular metric $g_{\widehat{V}}$ on the $\mathbb Q$-line bundle $K_{\widehat{V}} +\Delta_{\widehat V}$ with positive curvature by the same arguments as those provided in the proof of Theorem~\ref{logbig} where compactness plays no role. 
A priori, $g_{\widehat{V}}$ is only defined on $f_{\widehat V}(Z_{\rm reg})$ but since that open set has complement whose codimension is at least two in $\widehat V$, the metric extends canonically across $f_{\widehat V}(Z_{\rm sing})$. 

It remains to see that the curvature of $g_{\widehat{V}}$ satisfies the required lower bound \eqref{eqcontrolcurv}. Outside $\mathrm{Sing}(p) \cup D$, the maps $p\colon\Omega\to X$ and $f_{\widehat V}:Z\to \widehat V$ are étale covers. In particular, one has on that locus an equality of smooth forms 
$$\Theta(h_\Omega)=p^*\Theta(g_X) \quad \mbox{and} \quad \Theta(h_{Z_{\rm reg}})=f_{\widehat V}^*\Theta(g_{\widehat V}) $$
and the differential of $f_{\widehat V}$ induces an isomorphism 
$f_{\widehat V}^*:\Omega_{\widehat V,y}^{1,1} \overset{\sim}{\longrightarrow} \Omega_{Z,x}^{1,1}$ whenever $f_{\widehat V}(x)=y$. By commutativity of the diagram in Figure~\ref{fig:MD2}, one has that $j^*\Theta(h_{\Omega})= f_{\widehat V}^*(q\circ \sigma)^*\Theta(h_X)$ and therefore,  \eqref{eqcontrolcurv} follows from Lemma~\ref{lem:ineqcurv}. 
\end{proof}

\begin{rem}
As the reader will easily see, if we drop Assumption~\ref{funddomainassumption} (in particular if we only assume that $\Omega$ is a manifold of  bounded type), the same proof shows that $K_{\widehat{V}} + \Delta_{\widehat{V}}$ admits a singular metric with positive curvature, but we cannot obtain the bound \eqref{eqcontrolcurv} anymore. 
\end{rem}

\subsection{Statement of the criterion}
\label{sec:criterion}
In this section, we will state a hyperbolicity criterion for $X$, assuming it is now a \emph{compact, non-necessarily smooth} quotient. The precise assumption is as follows.

\begin{assumption}
\label{compactassumption}
Under the hypotheses of Assumption \ref{generalassumption}, we assume moreover that $X$ itself is a \emph{compact} complex space, \textsl{i.e.} $\overline{X} = X$.
\end{assumption}

We resume the notations of the previous section.

\begin{thm}
\label{thmsingmetric2}
Assume that $X = \X$ is as in Assumption \ref{compactassumption}, and let $\pi: \widehat{X} \longrightarrow X$ be a projective resolution of singularities of $X$, which exists accordingly to Lemma \ref{projective-resolution}. Let $\alpha > \frac{1}{C_0}$, and assume that the $\mathbb Q$-divisor
$$
L_\alpha = \pi^\ast (K_{X}+\Delta_X) - \alpha \Delta_{\widehat X} 
$$
is effective, where $\Delta_X$ is the covering divisor associated to $p:\Omega\to X$. Then
\begin{enumerate}
\item any subvariety $W \subseteq X$ such that $W \not\subset \pi (\mathbb B(L_\alpha)) \cup \mathrm{Sing}(p)$ is of general type.
\item any entire curve $f: \mathbb C \longrightarrow X$ has its image included in $\pi(\mathbb B(L_\alpha)) \cup \mathrm{Sing}(p)$.
\end{enumerate}
\end{thm}
\begin{proof}

Let us prove first the statement concerning subvarieties. Suppose that $W \subset X$ is a subvariety as in the theorem, and let $V$ be a resolution of singularities of $W$. Since the natural map $V \overset{q}{\longrightarrow} \overline{X}$ is generically immersive, the relative construction of Section \ref{functconst} can be applied, which yields a smooth bimeromorphic model $\widehat{V}$ of $V$ (which we can assume to be projective since $\widehat{X}$ is) and a map $r:\widehat V\to \widehat X$ as in Figure~\ref{fig:MD}.

\medskip

By Theorem \ref{thmsingmetric}, the $\mathbb Q$-divisor $\mathcal O(K_{\widehat{V}} + \Delta_{\widehat{V}} )$ admits a metric with positive curvature, and it is controlled as in \eqref{eqcontrolcurv} on $(\pi \circ r)^{-1}( X\setminus \mathrm{Sing}(p) )$.
By assumption, $W \not\subset \mathbb B(L_\alpha)$, so for $m$ large enough, there exists a section $$\sigma \in H^0 \left(\widehat{X}, m(\pi^\ast (K_X+\Delta_X) - \alpha \Delta_{\widehat X}) \right)$$ such that $\sigma|_{W} \neq 0$. 

\medskip

We denote by $\phi_X$ (resp. $\phi_{\widehat V}$) the psh weight on $K_X+\Delta_X$ (resp. $K_{\widehat V}+\Delta_{\widehat V}$) associated to the metric $g_X$ (resp. $g_{\widehat V}$). Next, we introduce the canonical singular weights $\phi_{\Delta_{\widehat X}}$ (resp. $\phi_{\Delta_{\widehat V}}$) on $\mathcal O_{\widehat X}(\Delta_{\widehat X})$ (resp. $\mathcal O_{\widehat V}(\Delta_{\widehat V})$) whose curvature current is $[\Delta_{\widehat X}]$ (resp. $[\Delta_{\widehat V}]$).

Because $p^*\phi_X$ is the weight associated to the Bergman metric on $\Omega$, the weight $\phi_X$ is locally bounded. One introduces the quantity $$F:=|\sigma|^2 e^{-m\pi^*\phi_X}e^{m\alpha \phi_{\Delta_{\widehat X}}},$$
it is a \textit{function} on $\widehat X$. Therefore, for any positive number $\beta>0$, the quantity
\begin{equation}
\label{phi}
e^{-\phi}:=(r^*F)^{-\beta/m}e^{-\phi_{\widehat V}+\phi_{\Delta_{\widehat V}}}
\end{equation}
defines a singular Hermitian metric on $K_{\widehat V}$. The first item in the theorem is a consequence of the following claim thanks to  standard results in pluripotential theory together with \cite{bou02}.

\begin{claim}
\label{claim1}
For $\beta \in \left( \frac{1}{\alpha}, C_0 \right)$, the following two properties hold: 
\begin{enumerate}
\item[$(i)$] The weight $\phi$ is locally bounded above
\item[$(ii)$] The weight $\phi$ is smooth and strictly psh on $(\pi \circ r)^{-1}(X\setminus \mathrm{Sing}(p))$. 
\end{enumerate}  
\end{claim}

\begin{proof}[Proof of Claim~\ref{claim1}]
To prove $(i)$, one observes that 
$$\phi=\Big(\phi_{\widehat V}+\frac\beta m \log |r^*\sigma|^2\Big)-\beta \cdotp (\pi \circ r)^*\phi_X+(\alpha\beta\cdotp \phi_{r^*\Delta_{\widehat X}}-\phi_{\Delta_{\widehat V}})$$
can be decomposed as the sum of two psh weights, one locally bounded weight and another weight which is psh thanks to Proposition~\ref{propeffective} together with the fact that $\alpha\beta>1$. In order to prove $(ii)$, one computes the curvature of the weight $\phi$ on $(\pi \circ r)^{-1}(X\setminus \mathrm{Sing}(p))$ as follows
\begin{align*}
dd^c \phi&= i\Theta(g_{\widehat V})-\beta \cdotp (\pi \circ r)^*i\Theta(g_X)+  \frac\beta m dd^c \log |r^*\sigma|^2\\
& \ge i\Theta(g_{\widehat V})-\beta \cdotp (\pi \circ r)^*i\Theta(g_X) \\
& \ge (C_0-\beta) \cdotp (\pi \circ r)^*i\Theta(g_X)
\end{align*}
where the last inequality follows from the inequality \eqref{eqcontrolcurv} in Theorem~\ref{thmsingmetric}.
\end{proof}

\medskip

The proof of the second point is very similar: we just have to perform the previous steps with $V = \mathbb C$ in a slightly more explicit manner, and then use the Ahlfors-Schwarz lemma (see \textsl{e.g.} \cite{dem12a}). 

Suppose then that there exists a non-constant holomorphic map $f: \mathbb C \longrightarrow X$ such that $f(\mathbb C) \not\subset \pi(\mathbb B(L_\alpha)) \cup \mathrm{Sing}(p) $. Now, if we perform the relative orbifold construction with $V = \mathbb C$, we see that $\widehat{V} \cong \mathbb C$, since this manifold admits a bimeromorphic map onto $\mathbb C$.

Moreover, since $\iota: \widetilde{V} \longrightarrow \Omega$ is non constant, we see that the universal cover of $\widetilde{V}$ must be isomorphic to the disk, and thus $i \Theta(h_{\widetilde{V}}) = h_{\widetilde{V}}^{-1}$. Pushing forward to $\widehat{V}$, we get that 
$$
i \Theta(g_{\widehat{V}}) = g_{\widehat{V}}^{-1}
$$
in restriction to the regular locus $V^\circ = \widehat{V} \setminus (f\circ\sigma  )^{-1}( \mathrm{Sing}(p))$.

Construct now the metric $h=e^{-\phi}$ on $\mathbb C = \widehat{V}$ using \eqref{phi}. By Claim~\ref{claim1} (ii), valid for $V$ pseudoconvex Kähler, we see that there exists a constant $\delta > 0 $ such that $i \Theta(h) \geq \delta i \Theta(g_{\widehat{V}})$ in \emph{restriction to $V^\circ$}. Now, we have, again in restriction to $V^\circ$: 
$$
\Theta(g_{\widehat{V}}) = g_{\widehat{V}}^{-1}
			\geq \frac{1}{\mathrm{sup}_{\widehat{X}} (r^*F)^{\frac{ \beta}{m}}} h^{-1}
$$
\noindent
Thus, in restriction to $V^\circ$, one gets
\begin{equation} \label{eqAS}
i \Theta(h) \geq C h^{-1}
\end{equation}
where $C = \frac{\delta}{\mathrm{sup}_{\widehat{X}} (r^*F)^{\frac{ \beta}{m}}}$.

\noindent
Finally, we see as in Claim~\ref{claim1} $(i)$ above that $\phi$ is locally bounded above near $\widehat V \setminus V^\circ$ and therefore $h=e^{-\phi}$ induces a positively curved metric on $K_{\widehat{V}}\simeq \mathcal O_{\mathbb C}$. This implies that \eqref{eqAS} holds \emph{everywhere on $\widehat{V} = \mathbb C$ in the sense of currents}. This is however absurd because of the Ahlfors-Schwarz lemma, cf. \cite[Lem.~3.2]{Dem95}.
\end{proof}

\bibliographystyle{amsalpha}
\bibliography{biblio2}
\end{document}